\documentclass[final]{amsart}
\usepackage{rjmacros4}
\begin{document}
\title{Weyl families of transformed boundary pairs}
\author{Rytis Jur\v{s}\.{e}nas}
\address{Vilnius University,
Institute of Theoretical Physics and Astronomy,
Saul\.{e}tekio ave.~3, 10257 Vilnius, Lithuania}
\email{rytis.jursenas@tfai.vu.lt}
\date{\today}
\begin{abstract}
Let $(\fL,\Gamma)$ be an isometric boundary
pair associated with a closed symmetric linear
relation $T$ in a Krein space $\fH$. Let $M_\Gamma$ be the
Weyl family corresponding to $(\fL,\Gamma)$.
We cope with two main topics. First,
since $M_\Gamma$ need not be (generalized) Nevanlinna,
the characterization of the closure and the adjoint
of a linear relation $M_\Gamma(z)$, for
some $z\in\bbC\setm\bbR$, becomes a nontrivial
task. Regarding $M_\Gamma(z)$ as the
(Shmul'yan) transform of $zI$ induced by $\Gamma$,
we give conditions for the equality in
$\ol{M_\Gamma(z)}\subseteq\ol{M_{\ol{\Gamma}}(z)}$
to hold and we compute the adjoint $M_{\ol{\Gamma}}(z)^*$.
As an application we ask when
the resolvent set of the main transform
associated with a unitary boundary pair
for $T^+$ is nonempty. Based on the criterion
for the closeness of $M_\Gamma(z)$ we
give a sufficient condition for the answer.
From this result it follows, for example,
that, if $T$ is a standard linear relation
in a Pontryagin space then the Weyl family $M_\Gamma$
corresponding to a boundary relation $\Gamma$ for $T^+$
is a generalized Nevanlinna family;
a similar conclusion is already known if $T$
is an operator.
In the second topic we characterize
the transformed boundary pair $(\fL^\prime,\Gamma^\prime)$
with its Weyl family $M_{\Gamma^\prime}$.
The transformation scheme is either
$\Gamma^\prime=\Gamma V^{-1}$ or
$\Gamma^\prime=V\Gamma$ with suitable
linear relations $V$.
Results in this direction
include but are not limited to:
a 1-1 correspondence between $(\fL,\Gamma)$ and
$(\fL^\prime,\Gamma^\prime)$; the formula for
$M_{\Gamma^\prime}-M_\Gamma$, for an
ordinary boundary triple and a standard unitary
operator $V$ (first scheme); construction
of a quasi boundary triple from an isometric
boundary triple $(\fL,\Gamma_0,\Gamma_1)$
with $\ker\Gamma=T$ and $T_0=T^*_0$
(second scheme, Hilbert space case).
\end{abstract}
\maketitle
\section{Introduction}
In this paper we presume that the
reader is familiar with the concept of an
abstract boundary value space of a symmetric
linear relation in a Krein space.
Particularly,
an operator is viewed as a linear relation via its graph.
For a modern treatment of linear relations
in Hilbert spaces the reader may refer to
\cite{Behrndt20,Hassi18,Hassi09,Hassi07,Hassi07b} and
for a classical theory of operators
in spaces with an indefinite metric one should look at
\cite{Azizov89,Azizov79,Iokhvidov76,Shmulyan74,Krein71,Ginzburg57,Gokhberg57,Iokhvidov56} and references therein.

To adapt notation and terminology,
a Krein space $\fH=(\fH,[\cdot,\cdot])$, or
a $J$-space, is a Hilbert space
$\fH=(\fH,\braket{\cdot,\cdot})$ equipped with
a fundamental symmetry $J$ and an indefinite
metric $[\cdot,\cdot]:=\braket{\cdot,J\cdot}$.
If several Krein spaces are considered,
we add a subscript $J=J_\fH$, and
similarly for an identity operator $I=I_\fH$,
scalar products, and norms.
A linear relation is referred to as a relation.

The main object of interest is
a \textit{boundary pair} (\bp\!\!),
$(\fL,\Gamma)$, where $\fL$ is a Hilbert space
and $\Gamma$ is an isometric relation
from a $\whJ_\fH$-space to a $\whJ^\circ_\fL$-space.
The fundamental symmetries in $\fH^2$
and $\fL^2$ are defined by
\[
\whJ_\fH:=\begin{pmatrix}
0 & -\img J_\fH \\ \img J_\fH & 0
\end{pmatrix}\quad\text{and}\quad
\whJ^\circ_\fL:=\begin{pmatrix}
0 & -\img I_\fL \\ \img I_\fL & 0
\end{pmatrix}\,.
\]
The (Krein space) adjoint $\Gamma^+$ is a
closed relation $\Gamma^+:=\whJ_\fH\Gamma^*\whJ^\circ_\fL$, where
$\Gamma^*$ is the (Hilbert space) adjoint of
$\Gamma$; the definition generalizes to an
arbitrary relation as usual. With this notation
$\Gamma$ is isometric iff $\Gamma\subseteq\Gamma_\#$,
where $\Gamma_\#:=(\Gamma^+)^{-1}$ and the
inverse is in the sense of linear relations
unless stated otherwise.

Suppose a closed relation $T=\ker\Gamma_\#$
is symmetric in a Krein space $\fH$, that is,
$T\subseteq T^+$.
Then a \bp $(\fL,\Gamma)$ is called an
\textit{isometric boundary pair} (\ibp\!\!) for $T^+$.
If $\Gamma$ is (essentially) unitary,
\ie ($\ol{\Gamma}=\Gamma_\#$)
$\Gamma=\Gamma_\#$, then $(\fL,\Gamma)$ is
said to be an
(\textit{essentially}) \textit{unitary boundary pair}
(abbreviated \eeubp\!\!) for $T^+$.
If $(\fL,\Gamma)$ is a \ubp for $T^+$ then
$\Gamma$ is also termed a \textit{boundary relation}
(\br\!\!) for $T^+$.

Let $T$ be as above, then $\fT:=JT$
is a closed symmetric relation in a Hilbert
space $\fH$, \ie $\fT\subseteq \fT^*=JT^+$.
An \ibp for $T^+$ always exists, since
the existence of a \ubp for $T^+$ is established
via a Hilbert space unitary operator from the existence of
a \ubp for $\fT^*$ \cite{Behrndt11,Derkach06}.

With $\Gamma$ one associates two relations,
$\Gamma_0:=\pi_0(\Gamma)$ and $\Gamma_1:=\pi_1(\Gamma)$,
in a usual way by using
\begin{align*}
\pi_0&\co\fH^2\times\fL^2\lto\fH^2\times\fL\,,\quad
(\whf,\whl)\mapsto(\whf,l)\,,\quad
\whl=(l,l^\prime)\,,
\\
\pi_1&\co\fH^2\times\fL^2\lto\fH^2\times\fL\,,\quad
(\whf,\whl)\mapsto(\whf,l^\prime)\,.
\end{align*}
The associated relations
$T_0:=\ker\Gamma_0$ and $T_1:=\ker\Gamma_1$
are symmetric in a Krein space $\fH$.
Similarly one puts
$(\ol{\Gamma})_i:=\pi_i(\ol{\Gamma})$
and $\mathring{T}_i:=\ker(\ol{\Gamma})_i$ for
$i\in\{0,1\}$.

If $\Gamma$ is an operator,
a \bp $(\fL,\Gamma)$ is referred to as a
\textit{boundary triple} (\bt\!\!).
In this way an \ibp $(\fL,\Gamma)$ for $T^+$
reduces to an \textit{isometric boundary triple} (\ibt\!\!)
$\Pi_\Gamma=(\fL,\Gamma_0,\Gamma_1)$ for $T^+$.
If $\Gamma$ is a closable operator then
$\Pi_{\ol{\Gamma}}=(\fL,(\ol{\Gamma})_0,(\ol{\Gamma})_1)$
is also an \ibt for $T^+$.

An \ibp $(\fL,\Gamma)$ for $T^+$ is
a generalization of an
\textit{ordinary boundary triple} (\obt\!\!)
$\Pi_\Gamma$ for $T^+$,
since a \ubp $(\fL,\Gamma)$ with a surjective $\Gamma$
is an \obt $\Pi_\Gamma$.
The systematic study of \obt\!\!'s in a Hilbert space
setting traces to \cite{Derkach91}
and to earlier references therein, starting from
\cite{Calcin39}.
It goes without saying that the \obt\!\!'s
play a central role in studying differential
operators with boundary conditions, where
they appear via the classical Green's formula.

A crucial advantage
of an \ibp over an \obt is that in the former
case $T$ need not have equal defect numbers,
by which we mean the defect numbers of $\fT$.
That is, $T$ has an \obt iff $\fT$ has equal
defect numbers
\cite{Behrndt11,Derkach06,Derkach95a,Jonas95,Shmulyan74}.
The so-called
$D$-\bt\!\!'s, which generalize
\obt\!\!'s to the case of not necessarily equal
defect numbers, or \obt\!\!'s for dual pairs of
operators are not discussed in this paper.
For these types of \bt\!\!'s we refer the reader to
\cite{Mogilevski11,Mogilevski09,Mogilevski06,Hassi13,Malamud02,Malamud97}.

A standard formulation of a \bp for $T^+$,
see \eg
\cite{Derkach21,Derkach17,Derkach12,Behrndt11,Derkach06,Derkach95,Derkach94},
formally differs from that given above but is
nevertheless equivalent to it. This is because,
by standard definition, for an
\ibp $(\fL,\Gamma)$ associated with a closed symmetric
relation, say $A$, it necessarily holds
$A=T(=\ker\Gamma_\#)$.

As in a Hilbert space setting,
with an \ibp $(\fL,\Gamma)$ for $T^+$ one
associates the \textit{Weyl family} $M_\Gamma$ and
the \textit{$\gamma$-field} $\gamma_\Gamma$ as follows
(\cite{Behrndt11,Derkach17}):
\[
M_\Gamma(z):=\Gamma(\whfN_z(A_*))\,,\quad
A_*:=\dom\Gamma\,,\quad
\whfN_z(A_*):=A_*\mcap zI\,,
\]
\[
\gamma_\Gamma(z):=
\{(l,f_z)\vrt f_z\in\fN_z(A_*)\,;\,
(\whf_z,l)\in\Gamma_0 \}\,,\quad
\whf_z=(f_z,zf_z)\,.
\]
For an \ibt $\Pi_\Gamma$
the corresponding Weyl family and $\gamma$-field
reduce to
\[
M_\Gamma(z)=\{(\Gamma_0\whf_z,\Gamma_1\whf_z)\vrt
\whf_z\in\whfN_z(A_*) \}\,,
\]
\[
\gamma_\Gamma(z)=\{(\Gamma_0\whf_z,f_z)\vrt
f_z\in\fN_z(A_*) \}\,.
\]
In case an \ibt is an \obt
the main difference between these
definitions and those \eg in
\cite{Derkach99,Derkach95,Derkach95a,Derkach94}
is that here
$z$ takes values from $\bbC_*:=\bbC\setm\bbR$, namely,
the definitions do not require
a self-adjoint relation
$T_0$ ($=T^+_0$) to have a nonempty resolvent set
$\res T_0$. Since $\bbC_*\subseteq\res T_0$
need not hold, now $M_\Gamma(z)$,
$\gamma_\Gamma(z)$ need not be (bounded) operators.
If, however, $M_\Gamma(z)$ is an operator
for some $z\in\bbC_*$, then $M_\Gamma$ is
termed the \textit{Weyl function}.

In a Hilbert space case,
Weyl families are fundamental
in classifying Nevanlinna (sub)families
via specific \bp\!\!'s \cite{Derkach17},
since the Weyl family corresponding to a \ubp
$(\fL,\Gamma)$ for $T^*$
is a Nevanlinna family \cite{Derkach06}.
The famed \ubp\!\!'s are as follows:
\begin{itemize}
\item[$*$]
$ES$-generalized \bp
(\ubp with $\ol{T_0}=T^*_0$);
\item[$*$]
$S$-generalized \bp
($ES$-generalized \bp with $T_0=\ol{T_0}$);
\item[$*$]
$B$-generalized \bp
($S$-generalized \bp with a surjective $\Gamma_0$).
\end{itemize}
An \ibp is called
an $AB$-generalized \bp\!\! if
$T_0=T^*_0$ and $\Gamma_0$ has dense range.
Particularly,
an $AB$-generalized \bp with $\Gamma$ having
dense range is a \bt known as a
\textit{quasi boundary triple} (\qbt\!\!).
The notion of a \qbt was introduced earlier in
\cite{Behrndt07}.

In a Pontryagin space setting
(\eg $\dim(I-J)\fH<\infty$),
the Weyl family $M_\Gamma$ corresponding to a \ubp
$(\fL,\Gamma)$ for $T^+$---where
a \br $\Gamma$ is either minimal
(\ie $T$ is simple, hence an operator)
or $\Gamma$ need not be minimal, but then
$T$ is assumed to be an operator---is a
generalized Nevanlinna family
\cite{Behrndt11}. As pointed out in
\cite{Behrndt11}, the transformation results
do not make it possible to derive \textit{specific}
analytic properties
of a Weyl family of $T$ from a Weyl family
of $\fT$, or vice verse. As the following
example (Example~\ref{exam:fT}, Theorem~\ref{thm:IUBP3})
illustrates, \textit{some} of the properties
can nonetheless be grasped:
Let $T$ be a densely defined, closed, symmetric
operator in a Krein space $\fH$,
let $\Pi_\Gamma$ be an \obt for $T^+$
with Weyl function $M_\Gamma$ and $\gamma$-field
$\gamma_\Gamma$. The Weyl function $M_{\Gamma^\prime}$
corresponding to an \obt $\Pi_{\Gamma^\prime}=\Pi_\Gamma$
for $\fT^*$
(here $\dom\Gamma=T^+$ is identified with
$\dom T^+=\dom\fT^*$) is given by
\begin{equation}
M_{\Gamma^\prime}(z)-M_\Gamma(z)=
z\Gamma_1(\fT_0-z)^{-1}(I-J)\gamma_\Gamma(z)
\label{eq:fTex}
\end{equation}
for $z\in\res T_0\mcap\res\fT_0$; $\fT_0:=J T_0=\fT^*_0$.

Although in this paper we concentrate on ``kinematic''
rather than analytic properties of Weyl families,
some closeness and existence results are included too,
as explained next.
\subsection*{Synopsis. Main results}
We are interested in two seemingly
different but actually closely related topics,
since nearly all results presented here have their
common roots in the equality $(V\hsum W)^+=V^+\mcap W^+$,
\eg \cite{Hassi20,Hassi09,Azizov89}, and in
the conditions for:
the closeness of the componentwise sum $V\hsum W$
of relations;
the equality in $(RX)^+\supseteq X^+R^+$
for appropriate relations;
the equality in $V(T)^+\supseteq V_\#(T^+)$,
$V_\#:=(V^+)^{-1}$,
for appropriate relations.
Results of different kind for obtaining
\eg a unitary operator
from another unitary operator can be found in
\cite{Wietsma12}, but they
have different purpose, not directly related to ours.
For the existence results on factorization
of relations the reader may look at \cite{Popovici13}.

After a preliminary Section~\ref{sec:knl},
in Section~\ref{sec:prelim2} we
consider a relation $T$ in a Krein space $\fH$
and its (Shmul'yan) transform $V(T)$ induced
by a relation $V$ from a $\whJ_\fH$-space
to a $\whJ_\cH$-space. We examine
the adjoint $V(T)^+$ based on a general hypothesis
that there is the equality in
$\ol{V\vrt_T}\subseteq\ol{V}\vrt_{\ol{T}}$,
where $V\vrt_T$ stands for the domain restriction to
$T$ of $V$.

Let $T$ be a closed symmetric relation in
a Krein space $\fH$, let $(\fL,\Gamma)$ be
an \ibp for $T^+$ with Weyl family $M_\Gamma$;
then $(\fL,\ol{\Gamma})$ is also an \ibp for $T^+$
with Weyl family $M_{\ol{\Gamma}}$.
Regarding $M_\Gamma(z)$ as the Shmul'yan transform of
$zI$ induced by $\Gamma$, in Section~\ref{sec:transW}
we apply the results from the previous section
for the characterization of the closeness of
$M_\Gamma(z)$ and the computation of its
adjoint $M_\Gamma(z)^*$.
In Theorem~\ref{thm:rrz}
we prove that, for all $z\in\bbC_*$ such that
$\ran(A_\star-zI)$ is a subspace, where
$A_\star:=\dom\ol{\Gamma}$, the adjoint
$M_{\ol{\Gamma}}(z)^*$ is the Shmul'yan transform
of $\ol{z}I$ induced by $\Gamma_\#$; if also
$\ran(A_\#-\ol{z}I)$ is a subspace, where
$A_\#:=\dom\Gamma_\#$, then $M_{\ol{\Gamma}}(z)$
is a closed relation. Here and below, by a subspace
we mean a closed linear subset.
We apply the result to the following cases:
$T$ is a standard relation \cite{Azizov03} in a Pontryagin
space (Corollary~\ref{cor:rrzz},
Examples~\ref{exam:Pstan}, \ref{exam:Pstan2})
and $\Gamma=\ol{\Gamma}$ is a \br for $T^+$, in which case
$M_\Gamma$ is a generalized Nevanlinna family
(\cf \cite[Theorem~4.8]{Behrndt11});
$T=\ker\ol{\Gamma}$ has finite defect numbers
(Corollary~\ref{cor:rrz}).

Conditions for the equality in
$\ol{M_\Gamma(z)}\subseteq\ol{M_{\ol{\Gamma}}(z)}$
to hold are also discussed.
The topic is important, for example,
in a Hilbert space setting, provided
one asks whether or not the Weyl family
$M_\Gamma$ corresponding to
an \eubp $(\fL,\Gamma)$ for $T^*$
satisfies $\ol{M_\Gamma(z)}=M_{\ol{\Gamma}}(z)$,
\ie whether or not $z\mapsto\ol{M_\Gamma(z)}$
is a Nevanlinna family. A sufficient
condition for an affirmative answer
is that $\ol{\Gamma\vrt_{zI}}=
\ol{\Gamma}\vrt_{zI}$, and by
Theorem~\ref{thm:PzfN} the latter equality holds
if $\Gamma$ is a suitable closable operator; take,
for example, an \ibt $\Pi_\Gamma$ for $T^*$ such that
$\fN_z(A_*)=\fN_z(T^*)\mcap\ran(A_*+\ol{z}^{\;-1}I)$
(particularly $\fN_z(A_*)=\fN_z(T^*)$ will do)
and such that
$\Pi_{\ol{\Gamma}}$ is an \obt for $T^*$.

The transformation results for \bp\!\!'s
can be useful when dealing with some existence issues.
For example, there exists an \obt $\Pi_\Gamma$ for $T^+$
with the Weyl family $M_\Gamma$ such that $M_\Gamma(z)$
is an operator for some $z$ iff
there exists an \obt $\Pi_{\Gamma^\prime}$ for $T^+$
with the Weyl family $M_{\Gamma^\prime}$ such that
$M_{\Gamma^\prime}(z)$ is an injective relation for some $z$.
We emphasize that, in a general Krein space setting,
$M_\Gamma(z)$ is generally
neither an operator nor injective, since
the multivalued part and the kernel are
represented by
\[
\mul M_\Gamma(z)=\Gamma_1(\whfN_z(T_0))\,,\quad
\ker M_\Gamma(z)=\Gamma_0(\whfN_z(T_1))
\]
and they need not be trivial even when $\Gamma$
is an operator.

A more involved example concerns the existence
of a nonempty resolvent set of the main transform
associated with a \bp The topic is important
at least in that the nonemptyness of the resolvent
set allows one to characterize a number of
operator functions by applying the methods already
developed in a Hilbert space setting
(where one heavily relies on the fact that $\bbC_*$
is a subset of the resolvent set of a self-adjoint
relation) \cite{Derkach06}.

The main transform of a \br $\Gamma$ for $T^+$,
here denoted by $\Gamma\mapsto\cJ(\Gamma)$,
establishes the correspondence between
\bp\!\!'s and exit space extensions to
$\fH\times\fL$ of $T$, and was first introduced
in \cite{Derkach06} in the context of Hilbert spaces;
see \eqref{eq:mrTG} for a precise definition.
If $\Gamma=\Gamma_\#$ then
$\cJ(\Gamma)$ is a self-adjoint relation
in a Krein space $\fH\times\fL$
(see also \cite[Section~3.2]{Behrndt11}), but generally
it can have an empty resolvent set
(\eg $\cJ(\Gamma)=I\times I$ with
$\fH=\fL=\bbC$ and $J=-1$ in
\cite[Example~3.7]{Behrndt11}).
When $T$ is a closed symmetric
operator in a Pontryagin space
and $\Gamma$ is a \br for $T^+$,
it is shown in \cite[Proposition~3.12]{Behrndt11}
that there exists a \br for $T^+$ whose
main transform has a nonempty resolvent set.
For a closed symmetric relation $T$ in a
Krein space
we derive a similar result in Theorem~\ref{thm:resTG},
namely, if there is a \br $\Gamma$
for $T^+$ with a nonempty set of
$z\in\bbC_*\setm\sigma_p(T)$ such that
$\ol{z}\in\bbC_*\setm\sigma_p(T)$ and such that
$\ran(T-\omega I)$ is a subspace
for both $\omega=z$ and $\omega=\ol{z}$, then
$\Gamma$ is in 1-1 correspondence with
a \br $\Gamma^\prime$ for $T^+$ such that
$\res\cJ(\Gamma^\prime)\neq\emptyset$.
The proof relies on Theorem~\ref{thm:rrz},
the canonical decomposition of a closed relation
$M_\Gamma(z)$, and a 1-1 correspondence
between \ubp\!\!'s $(\fL,\Gamma)$ and
$(\fL,\Gamma^\prime:=V\Gamma)$ for $T^+$,
with a standard unitary operator $V$ in
a $\whJ^\circ_\fL$-space.

Another type of problems that can be solved
with the help of transformations and that
interests us most
is the characterization of a \bp
provided one is given another \bp
A transformation can be convenient in that
certain properties of \bp\!\!'s are easier
to handle in one ``coordinate system'' than in another.
We restrict ourselves
to cases (1) $\Gamma^\prime:=\Gamma V^{-1}$ and (2)
$\Gamma^\prime:=V\Gamma$, where $V$ is a suitable
relation in each case separately.

In the context of transformed boundary value spaces
we have found in the literature
these types of transformations:

$-$
With a standard unitary operator
$V=U_{J}:=\bigl(\begin{smallmatrix}I & 0 \\
0 & J
\end{smallmatrix} \bigr)$, case (1)
is extensively studied in \cite{Behrndt11}
in establishing correspondences between
a symmetric relation $T$ in a Krein space
and a symmetric relation $\fT$ in a Hilbert
space;
see also Examples~\ref{exam:fT}, \ref{exam:fT2}.

$-$
With a standard unitary operator of a general
form,
$V=\bigl(\begin{smallmatrix}A & B \\
C & D
\end{smallmatrix} \bigr)$, case (2)
can be called a classical one, see \eg
\cite{Behrndt20,Derkach12,Derkach09} and references
there.

$-$
With an injective operator
$V=\bigl(\begin{smallmatrix}G^{-1} & 0 \\
EG^{-1} & G^*
\end{smallmatrix} \bigr)$, $V\subseteq V_\#$,
case (2) is extensively studied in \cite{Derkach17},
see also Section~\ref{sec:exV}. This type of
$V$ has direct applications to
boundary value problems of differential
operators.

Our main results in this direction are
as follows.

\textit{Transformation scheme $\Gamma\lto\Gamma V^{-1}$}
(Section~\ref{sec:GV}):

A 1-1 correspondence
between an \ibp\!/\eeubp $(\fL,\Gamma)$ for
$T^+$ and an \ibp\!/\eeubp $(\fL,\Gamma^\prime)$ for
the Shmul'yan transform $V(T^+)$ of $T^+$, with
a suitable unitary relation $V$,
is established
in Theorem~\ref{thm:IUBP} and is a generalization of
\cite[Proposition~3.3]{Behrndt11}.
We then restrict ourselves to a
standard unitary operator of the
matrix form
$V=\bigl(\begin{smallmatrix}
A & B \\ C & D \end{smallmatrix}\bigr)$ and
a densely defined (closed, symmetric)
operator $T$. Due to the rich inner structure of
$V$, in case of the \obt\!\!'s
we are able to derive the formula for
$M_{\Gamma^\prime}-M_\Gamma$ similar to that
in \eqref{eq:fTex}; see Theorem~\ref{thm:IUBP3}.
The formula is convenient in answering
the question, similar to the classical one in
\cite{Langer77} but now in a Krein space setting.
Namely, under what circumstances there is the equality
$M_{\Gamma^\prime}=M_\Gamma$ ?
(Corollaries~\ref{cor:Delta0}, \ref{cor:Delta0b}).
We point out that in the present transformation scheme
the \obt\!\!'s
$\Pi_\Gamma$ and $\Pi_{\Gamma^\prime}$
are associated with generally different
operators, $T$ and its linear fractional
transformation $T^\prime=\phi_V(T)$
(Section~\ref{sec:LFT}), respectively,
but in the special
case it is possible to achieve that
$T^\prime=T$, as in the model for a
scaled \obt\!\!, in Section~\ref{exam:iota}.

\textit{Transformation scheme $\Gamma\lto V\Gamma$}
(Section~\ref{sec:VG}):

$\circ$
An \ibp $(\cH,\Gamma^\prime)$ for $T^+$ is constructed
from an \ibp ($\fL,\Gamma$) for $T^+$
with the help of an isometric relation $V$
(Theorem~\ref{thm:IBP-0}).

$\circ$
Relying partially on the previous result,
a 1-1 correspondence
between an \ibp\!/\eeubp $(\fL,\Gamma)$ for
$T^+$ and an \ibp\!/\eeubp $(\cH,\Gamma^\prime)$ for
$T^+$ is established
in Theorem~\ref{thm:IUBP2xxcor}
with the help of a suitable unitary relation $V$.

$\circ$
If $(\fL,\Gamma)$ is a \ubp for $T^+$
and a unitary relation $V$ satisfies
$\dom V\subseteq\ran\Gamma$,
in Theorem~\ref{thm:GunTp} we show that
$(\cH,\Gamma^\prime)$ is an \ibp for
the adjoint $T^{\prime\,+}$ of
a closed symmetric relation
$T^\prime=\Gamma^{-1}(\mul V^+)$.
It is seen that $T^\prime\supseteq T$,
with the equality if \eg $\dom V$, and then
$\ran\Gamma$, is dense.

$\circ$
From the previous theorem, by letting
a \ubp $(\fL,\Gamma)$ reduce to an \obt
$\Pi_\Gamma$ for $T^+$, it is possible to
achieve, in case $V$ is a suitable operator,
that the transformed \ibp $(\cH,\Gamma^\prime)$,
now the \ibt $\Pi_{\Gamma^\prime}$
for $T^{\prime\,+}$, has the properties
that $T^\prime_0:=\ker\Gamma^\prime_0=T_0$
is a self-adjoint relation and
$\Gamma^\prime$ has dense range; for $J=I$,
this is nothing else but a \qbt for $T^{\prime\,*}$
(Proposition~\ref{prop:VVV}). With the help
of an additional Lemma~\ref{lem:VVV}, for which
$T^\prime=T$, we are allowed to replace
an \obt $\Pi_\Gamma$ by an
\ibt $\Pi_\Gamma$ for $T^+$ such that $T_0=T^+_0$
and $\ker\Gamma=T$.
As an application,
the construction of a \qbt $\Pi_{\Gamma^\prime}$ for
$T^*$ from an \ibt $\Pi_\Gamma$
for $T^*$ is shown in Theorem~\ref{thm:VVV}.
An injective isometric operator $V$
is taken from \cite[Section~3.4]{Derkach17}.
Thus it is possible to construct a \qbt from \eg
an $AB$-generalized \bt\!\!, and then from its
restricted versions ($S$-, $B$-generalized, etc. \bt\!\!'s).
In \cite[Theorems~4.4, 4.11]{Derkach17},
an $AB$-generalized \bp
is constructed from a $B$-generalized \bp (and vice verse)
and a $B$-generalized \bp is constructed from an \obt
\subsection*{Notation}
The sets of relations, closed relations,
and closed and symmetric relations in a Krein
space $\fH$ are denoted by
$\scrC(\fH)$, $\wtscrC(\fH)$, and $\wtscrC_s(\fH)$,
respectively; if $\fH$ is a Hilbert space then
$T\in\wtscrC_s(\fH)$ is assumed symmetric with respect
to the Hilbert space metric.
If $T\in\wtscrC_s(\fH)$ and $\wtT\in\scrC(\fH)$
and $T\subseteq\wtT\subseteq T^+$
then we write $\wtT\in\mrm{Ext}_T$
(proper extension).

Other notations are rather standard.
Thus,
the set of bounded (continuous) everywhere
defined on $\fH$ operators from a Hilbert space
$\fH$ to a Hilbert space $\cH$ is abbreviated
as $[\fH,\cH]$; $[\fH,\fH]\equiv[\fH]$.
The domain, the range, the kernel,
the multivalued part of $T\in\scrC(\fH)$ are denoted
by $\dom T$, $\ran T$, $\ker T$, $\mul T$.

The symbols $\sigma(T)$,
$\sigma_p(T)$, $\res T$, $\reg T$
stand for the spectrum, the point spectrum,
the resolvent set (the set of regular points),
the regularity domain (the set of points of regular type)
of $T$. The eigenspace
$\fN_z(T):=\ker(T-zI)$, $z\in\bbC$.
We write $\sigma^0_p(T)$ for the intersection
$\bbC_*\mcap\sigma_p(T)$, where
$\bbC_*:=\bbC\setm\bbR=\bbC_+\mcup\bbC_-$,
$\bbC_\pm:=\{z\in\bbC\vrt\pm\Im z>0 \}$.

Let $\fH$ be a linear set and
$T$, $R\in\scrC(\fH)$.
The \textit{operatorwise sum} $T+R$ consists of
$(f,f^\prime+f^{\prime\prime})$
such that $(f,f^\prime)\in T$ and $(f,f^{\prime\prime})\in R$.
The \textit{componentwise sum} $T\hsum R$ consists
of $(f+g,f^\prime+g^\prime)$ such that
$(f,f^\prime)\in T$ and $(g,g^\prime)\in R$;
if $\fH$ is a Hilbert space and in addition
$\braket{f,g}+\braket{f^\prime,g^\prime}=0$
then $T\hsum R$ is written as $T\hop R$.
The definition of the sum extends
naturally to relations from
a linear set $\fH$ to a linear set $\cH$.
\section{Some known lemmas}\label{sec:knl}
Here we restate some known composition properties of
relations.
\begin{lem}{\cite[Lemma~2.9]{Derkach09}}\label{lem:Derk}
Let $\fK_i$, $i\in\{0,1,2,3\}$, be Krein spaces,
let $R$ be a closed relation $\fK_1\lto\fK_2$,
$X$ a relation $\fK_0\lto\fK_1$,
$Y$ a relation $\fK_2\lto\fK_3$.
\begin{itemize}
\item[$\circ$]
If $\dom R$ is closed and $\ran X\subseteq\dom R$
then $(RX)^+=X^+R^+$.
\item[$\circ$]
If $\ran R$ is closed and $\dom Y\subseteq\ran R$
then $(YR)^+=R^+Y^+$.
\end{itemize}
\end{lem}
The next lemma is a special case of \cite[Corollary~1]{Tarcsay20}.
\begin{lem}{\cite[Proposition~2.1]{Popovici13},
\cite[Corollary~2]{Tarcsay20}}\label{lem:Pop}
Let $\fH$, $\cH$ be Hilbert spaces,
$L\subseteq R$ relations $\fH\lto\cH$.
Equivalent are:
\begin{itemize}
\item[$\circ$]
$L=R$.
\item[$\circ$]
$\dom L=\dom R$ and $\mul L=\mul R$.
\item[$\circ$]
$\ran L=\ran R$ and $\ker L=\ker R$.
\end{itemize}
\end{lem}
Let $\fH$, $\cH$ be Krein spaces,
$V$, $W$ relations from a $\whJ_\fH$-space to a
$\whJ_\cH$-space (abbreviated as $\whJ_\fH\lto
\whJ_\cH$).
Variants of the next lemma are known
(\eg \cite[Lemma~2.10]{Hassi09}).
\begin{lem}\label{lem:VWorth}
The following three componentwise sums are
closed (or not) only simultaneously:
$\ol{V}\hsum\ol{W}$, $V^\bot\hsum W^\bot$,
$V^+\hsum W^+$.
\end{lem}
We use that the orthogonal complement
in $\fH^2\times\cH^2$ of $V$
satisfies the equality
$V^\bot=-\whJ_\cH V_\#\whJ_\fH$, where as previously
(and in what follows)
$V_\#:=(V^+)^{-1}$, and the Krein space
adjoint $V^+=\whJ_\fH V^*\whJ_\cH$.

Recall (\eg \cite[Lemma~2.6]{Hassi09}, \cite[Section~1.7]{Azizov89})
that
\[
(\ol{V}\mcap\ol{W})^+=\ol{V^+\hsum W^+}\,.
\]
It is essential that we use the closures on the
left hand side.
\section{Shmul'yan transform}\label{sec:prelim2}
Let $\fH$, $\cH$ be Krein spaces,
let $T$ be a relation in $\fH$, and let $V$
be a relation from a $\whJ_\fH$-space to a
$\whJ_\cH$-space.
The relation
$V(T)$ in $\cH$ is the linear set of
$\whh\in\cH^2$ such that $(\exists\whf\in T)$
$(\whf,\whh)\in V$.
Following \cite[Definition~3.9]{Derkach09}
$V(T)$ is called the \textit{Shmul'yan transform} of $T$
induced by $V$. Let $V\vrt_T$ be the domain
restriction to $T$ of $V$, \ie
\[
V\vrt_T:=V\mcap(T\times\cH^2)\,.
\]
The inclusion $T\subseteq\dom V$
need not hold, but clearly
$V\vrt_T=V\vrt_{T\mcap\dom V}$.
With the above notation
\[
V(T)=\ran(V\vrt_T)\,.
\]

For suitable $T$, $V$,
the inverse of the Shmul'yan transform
is the Shmul'yan transform. To see this,
let $T^\prime:=V(T)$. By using
\[
V^{-1}V=I_{\dom V}\hsum(\{0\}\times\ker V)
\]
we get that
\[
V^{-1}(T^\prime)=(T\mcap\dom V)\hsum\ker V\,.
\]
Therefore, if
\begin{equation}
\ker V\subseteq T\subseteq\dom V
\label{eq:kerVTdomV}
\end{equation}
then $T=V^{-1}(T^\prime)$ is the Shmul'yan
transform of $T^\prime$ induced by $V^{-1}$.

We characterize the adjoint
$V(T)^+$ based on a general hypothesis that there is
the equality in
\begin{equation}
\ol{V\vrt_T}\subseteq\ol{V}\vrt_{\ol{T}}\,.
\label{eq:ineVT}
\end{equation}
Necessary and sufficient conditions for the
equality in \eqref{eq:ineVT} to hold
are due to Lemma~\ref{lem:Pop};
note from Lemma~\ref{lem:Pop},
if there is the equality in \eqref{eq:ineVT} then
\[
\ol{V(T)}=\ol{\ol{V}(\ol{T})}\,.
\]
In connection with related results in
\cite{Wietsma12} we give two other criteria
in Lemmas~\ref{lem:Wie}, \ref{lem:clVTan0001}.

The next Lemma~\ref{lem:Torth} generalizes \eg
\cite[Proposition~3.10(i)]{Derkach09},
\cf \cite[Lemma~3.8]{Wietsma12}, and it relies on
the following hypotheses:
\begin{enumerate}[label=\text{(V.\arabic*)},ref=(V.\arabic*)]
\item\label{item:V}
$\ol{T}\hsum\dom \ol{V}\in\wtscrC(\fH)$.
\item\label{item:Vp}
$T^+\hsum\dom V_\#\in\wtscrC(\fH)$.
\end{enumerate}
Equivalent representations of \ref{item:V} and
\ref{item:Vp} are as follows.
Let
\[
U:=V_\#\hsum(T^+\times\{0\})
\]
then (by Lemma~\ref{lem:VWorth})
\begin{itemize}
\item[$\circ$]
\ref{item:V} iff $U=\ol{U}$.
\item[$\circ$]
\ref{item:Vp} iff $\dom U=\ol{\dom}U$.
\end{itemize}
Moreover
\begin{itemize}
\item[$\circ$]
If $\dom \ol{V}\in\wtscrC(\fH)$ then
\ref{item:V} iff
$T^+\hsum\ker V_\#\in\wtscrC(\fH)$.
\item[$\circ$]
If $\dom V_\#\in\wtscrC(\fH)$ then
\ref{item:Vp} iff
$\ol{T}\hsum\ker \ol{V}\in\wtscrC(\fH)$.
\end{itemize}
The connection between
\ref{item:V} and \ref{item:Vp} is as follows.
\begin{itemize}
\item[$\circ$]
If \ref{item:V} then
\ref{item:Vp} iff $\ol{V}(\ol{T})\in\wtscrC(\cH)$.
\item[$\circ$]
If \ref{item:Vp} then
\ref{item:V} iff $V_\#(T^+)\in\wtscrC(\cH)$.
\end{itemize}
For the first statement,
let $U_1:=\cH^2\times\dom U$,
then $U_1=\ol{U_1}$ iff \ref{item:Vp}.
Since \ref{item:V} holds, $U=\ol{U}$, so
$U_1=\ol{U_1}$ iff $\dom U^+=\ol{V}(\ol{T})$ is closed.
The second statement is the inverse variant of
the first one.

In the above one of the arguments uses that
$\dom \ol{V}$ is closed iff so is $\dom V^+$.
By replacing $V$ by $V^{-1}$, $\ran \ol{V}$
is closed iff so is $\ran V^+$;
\eg \cite[Proposition~2.1]{Sorjonen79},
\cite[Lemma~2.1]{Hassi20}.
\begin{lem}\label{lem:Torth}
Let $T$, $V$ be such that
there is the equality in \eqref{eq:ineVT}.
\begin{enumerate}[label=\arabic*$^\circ$,
ref=\arabic*$^\circ$]
\item\label{item:Torth-1o}
If \ref{item:V} then
\[
V(T)^+=V_\#(T^+)\,.
\]
\item\label{item:Torth-2o}
If \ref{item:V} and \ref{item:Vp}
then $\ol{V}(\ol{T})\in\wtscrC(\cH)$.
\end{enumerate}
\end{lem}
\begin{proof}
\ref{item:Torth-1o}
Using $\ol{V\vrt_T}=\ol{V}\vrt_{\ol{T}}$,
$V(T)^+=\ker(V\vrt_T)^+$,
\ref{item:V}, Lemma~\ref{lem:VWorth} gives
\[
V(T)^+=\ker(V^+\hsum(\{0\}\times T^+))=
V_\#(T^+)\,.
\]
\ref{item:Torth-2o}
Since $V_\#\vrt_{T^+}$ is closed, an application
of \ref{item:Torth-1o} gives
$V_\#(T^+)^+=\ol{V}(\ol{T})$.
\end{proof}
The next two lemmas deal with the equality in \eqref{eq:ineVT}.
\begin{lem}\label{lem:Wie}
Let $V$ be closed.
Then $\ol{V\vrt_T}=V\vrt_{\ol{T\mcap\dom V}}$.
\end{lem}
\begin{proof}
Put $T_V:=T\mcap\dom V$; then $V\vrt_T=V\vrt_{T_V}$.
Let $(\whf,\whh)\in V\vrt_{\ol{T_V}}$.
Since $\whf\in\ol{T_V}$, there is a sequence
$(\whf_{n})\subseteq T_V$ such that
$\norm{\whf_{n}-\whf}_{\fH^2}\lto0$ as $n\lto\infty$.
Since $T_V\subseteq \dom V$,
there is a $\whh_{n}\in\ran (V\vrt_T)$
such that $(\whf_{n},\whh_{n})\in V\vrt_T$.
Since $V\vrt_T$ is contained in a subspace
$V\vrt_{\ol{T_V}}$,
$\whh_n$ converges in $\cH^2$-norm to some $\whh_*\in\cH^2$
as $n\lto\infty$. This shows
$(\whf,\whh_*)\in\ol{V\vrt_T}$, and then
$\whh_*-\whh\in\mul V$. Therefore,
for every $(\whf,\whh)\in V\vrt_{\ol{T_V}}$
there is a $\whxi\in\mul V$
such that
$(\whf,\whh+\whxi)\in\ol{V\vrt_T}$.
This shows
$V\vrt_{\ol{T_V}}\subseteq\ol{V\vrt_T}$,
and then
$V\vrt_{\ol{T_V}}=\ol{V\vrt_T}$.
\end{proof}
If $V$ is a standard unitary operator,
that is, $V\in[\fH^2,\cH^2]$ and
$V^{-1}=V^+\in[\cH^2,\fH^2]$
(see also \cite[Definition~2.5]{Derkach09}),
the equality in \eqref{eq:ineVT} can be verified by direct
computation:
$(V\vrt_T)^+$ consists
of $(\whh,\whg)\in\cH^2\times\fH^2$ such that
$\whg-V^{+}\whh\in T^+$. Now $(V\vrt_T)^+=U^{-1}$
gives the result.

If $T=\ol{T}$, one recovers
from the lemma the equality in \eqref{eq:ineVT}
by using
\[
(\ol{T\mcap\dom V})\mcap\dom V=T\mcap\dom V\,.
\]

The next criterion for the equality in \eqref{eq:ineVT}
to hold relies on hypothesis \ref{item:V} and the following
observation.
Let $P_T$ be an orthogonal
(with respect to the Hilbert space metric)
projection in $\fH^2$ onto $T\in\wtscrC(\fH)$. Then
\begin{equation}
P_{T}\dom V\subseteq\dom V
\quad\Leftrightarrow\quad
P_{T}\dom V=T\mcap\dom V
\label{eq:PTequiv}
\end{equation}
since
\begin{equation}
T\mcap\dom V=
P_T\dom V\mcap\dom V\,.
\label{eq:nftr}
\end{equation}

If one of the conditions in \eqref{eq:PTequiv} holds then
\begin{equation}
\dom V=(T\mcap\dom V)\hop(T^\bot\mcap\dom V)
\label{eq:domVT}
\end{equation}
where $T^\bot$ is the orthogonal complement
in $\fH^2$ of $T$.
For example, if $T\subseteq\dom V$ then
$P_T\dom V=T$; if $T\supseteq\dom V$
then $P_T\dom V=\dom V$.
Recall \cite{Allan98} that every projection has
a nontrivial invariant subspace.
\begin{lem}\label{lem:clVTan0001}
Let $T$ be closed and let
\ref{item:V} hold. If also one of the conditions in
\eqref{eq:PTequiv} holds then
\[
\ol{T\mcap\dom V}=\ol{T\mcap\dom\ol{V}}\,.
\]
Moreover, if $V$ is a closable operator
whose closure has closed domain then
there is the equality in \eqref{eq:ineVT}.
\end{lem}
\begin{proof}
First we show
\begin{equation}
V=V_T:=(V\vrt_T)\hsum(V\vrt_{T^\bot})\,.
\label{eq:v1}
\end{equation}
By definition $V_T\subseteq V$. Since
$\mul V_T=\mul V$, by Lemma~\ref{lem:Pop}
$V_T=V$ iff $\dom V_T=\dom V$. But the latter
equality holds by \eqref{eq:domVT}.

Next we show
\begin{equation}
(\ol{V}\vrt_{T})^+=(V\vrt_T)^+
\vrt_{\dom(V\vrt_{T^\bot})^+ }\,.
\label{eq:djet2}
\end{equation}
Using $U=\ol{U}$ one has
$(\ol{V}\vrt_{T})^+=U^{-1}$. On the other hand,
by \eqref{eq:v1}
\[
V^+=(V\vrt_T)^+\mcap(V\vrt_{T^\bot})^+
\]
and then
\[
(\ol{V}\vrt_{T})^+=
\bigl( (V\vrt_T)^+\mcap(V\vrt_{T^\bot})^+ \bigr)
\hsum(\{0\}\times T^+)\,.
\]
Since $\dom(V\vrt_T)\subseteq T$ implies
$\mul(V\vrt_T)^+\supseteq T^+$, \ie
$(V\vrt_T)^+\supseteq\{0\}\times T^+$,
one concludes that $\hsum$ is distributive
over $\mcap$:
\[
(\ol{V}\vrt_{T})^+=(V\vrt_T)^+\mcap U_2\,,
\quad
U_2:=(V\vrt_{T^\bot})^+
\hsum(\{0\}\times T^+)\,.
\]
Since
\[
(V\vrt_{T^\bot})^+\supseteq
(\ol{V}\vrt_{T^\bot})^+\supseteq
V^+\hsum(\{0\}\times(T^\bot)^+)
\]
and $(T^\bot)^+=(T^+)^\bot$
($=\whJ_\fH(T)$)
it follows that
\[
U_2\supseteq V^+\hsum(\{0\}\times\fH^2)=
\dom V^+\times\fH^2\,.
\]
This shows
\[
U_2=\dom(V\vrt_{T^\bot})^+\times\fH^2\,.
\]

By \eqref{eq:djet2}
\[
\mul(\ol{V}\vrt_{T})^+=\mul(V\vrt_T)^+
\quad\Rightarrow\quad
\ol{\dom}(\ol{V}\vrt_{T})=\ol{\dom}(V\vrt_T)
\]
which is the first statement of the lemma.

Again by \eqref{eq:djet2}
\[
\ol{V}\vrt_{T}=\ol{V\vrt_T}
\quad\Leftrightarrow\quad
\dom(V\vrt_T)^+\subseteq \dom(V\vrt_{T^\bot})^+\,.
\]
Therefore, if $\dom V^+=\cH^2$ then
there is the equality in \eqref{eq:ineVT}.
\end{proof}
\begin{rem}
From the above proof somewhat more general statements
can be extracted. For example,
under conditions of the lemma
there is the equality in \eqref{eq:ineVT}
iff $\dom(V\vrt_T)^+=\dom V^+$. But then
one needs additional criteria for the latter
equality to hold, unless $\dom V^+=\cH^2$.
Also, one could suppose in Lemma~\ref{lem:clVTan0001}
that $T$ is not necessarily closed
and then replace $T$ by $\ol{T}$
(resp. $P_T$ by $P_{\ol{T}}$) and add
an additional hypothesis
$T\mcap\dom V=\ol{T}\mcap\dom V$.
On the other hand, unless $T\subseteq\dom V$,
the hypotheses of Lemma~\ref{lem:clVTan0001}
are not optimal, as can be seen by letting
$V=\ol{V}$; \cf Lemma~\ref{lem:Wie}. As a rule,
in applications that we discuss $T\subseteq\dom V$
does not hold in general.
\end{rem}
We compare the three lemmas
with some related conclusions obtained in
\cite{Wietsma12} by assuming in addition that
a relation $V\co\whJ_\fH\lto\whJ_\cH$ is
isometric, $V\subseteq V_\#$. This is precisely
the case when we consider \bp\!\!'s in
subsequent sections.
\begin{rem}
If $V\subseteq V_\#$ and $T\subseteq\dom V$ then by
\cite[Lemma~3.8]{Wietsma12}
\begin{equation}
V(T)^+\mcap\ran V=
V(T^+\mcap\dom V)\;(=V(T^+))\,.
\label{eq:Wie2}
\end{equation}
Particularly, if $V=V_\#$, $\dom V$ is closed,
and $T$ satisfies \eqref{eq:kerVTdomV}, then
by \cite[Proposition~4.5, Eq.~(4.1)]{Wietsma12}
$V(T)^+=V(T^+)$, $\ol{V(T)}=V(\ol{T})$.
To compare with, if $V=V_\#$, $\dom V$ is closed,
and $T\subseteq\dom V$, then \ref{item:V} holds
and by Lemma~\ref{lem:Wie}
there is the equality in \eqref{eq:ineVT}.
Condition \ref{item:Vp} reads
$\ol{T}\hsum\ker V\in\wtscrC(\fH)$,
and is satisfied if $T$ satisfies
\eqref{eq:kerVTdomV}.
Suppose now $V\subseteq V_\#$, $\ran V_\#=\cH^2$,
and $T\subseteq\dom V$ is closed. Then
condition \ref{item:V} holds and
by Lemma~\ref{lem:clVTan0001}
there is the equality in \eqref{eq:ineVT};
hence $V(T)^+=V_\#(T^+)$. Suppose in addition
$\ran\ol{V}=\cH^2$, then $V_\#(T^+)=\ol{V}(T^+)$.
Indeed,
let $\whh\in V_\#(T^+)$, then
$(\exists\whf\in T^+\mcap\dom V_\#)$ $(\whf,\whh)\in V_\#$.
Since $\ran \ol{V}=\cH^2$, $(\exists\whg\in\dom \ol{V})$
$(\whg,\whh)\in \ol{V}$. This shows $\whg-\whf\in\ker V_\#$
($\subseteq T^+$), so $\whg\in T^+\mcap\dom \ol{V}$
and then $V_\#(T^+)\subseteq \ol{V}(T^+)$.
If moreover $V$ is surjective, then
by similar arguments $\ol{V}(T^+)=V(T^+)$,
in agreement with \eqref{eq:Wie2}.
\end{rem}
In applications to be following we mostly restrict
ourselves to the cases where $V$ satisfies
the hypotheses of Lemmas~\ref{lem:Wie}, \ref{lem:clVTan0001},
see \eg
Theorems~\ref{thm:PzfN}, \ref{thm:IUBP}, \ref{thm:GunTp}.
Otherwise we rely on a general hypothesis
$\ol{V\vrt_T}=\ol{V}\vrt_{\ol{T}}$;
see Theorem~\ref{thm:IBP-0}.
\section{The adjoint and closure of
Weyl family}\label{sec:transW}
Let $\fH$ be a Krein space (unless explicitly specified
otherwise), $\fL$ a Hilbert space.
Let $T\in\wtscrC_s(\fH)$, let $(\fL,\Gamma)$
be an \ibp for $T^+$ with Weyl family $M_\Gamma$.
By definition
\[
M_\Gamma(z)=\ran\Gamma_z\,,
\quad
\Gamma_z:=\Gamma\vrt_{\whfN_z(A_*)}=
\Gamma\vrt_{zI}\,,
\quad z\in\bbC_*\,.
\]
Let $A_\star:=\dom\ol{\Gamma}$; \cf
$A_*:=\dom\Gamma$.
Since the closure $\ol{\Gamma}$ is isometric
and moreover $\ol{A_\star}=\ol{A_*}=T^+$,
the pair $(\fL,\ol{\Gamma})$ is also an \ibp
for $T^+$ with the Weyl family
$M_{\ol{\Gamma}}$ given by
\[
M_{\ol{\Gamma}}(z)=\ran(\ol{\Gamma})_z\,,\quad
(\ol{\Gamma})_z:=\ol{\Gamma}\vrt_{\whfN_z(A_\star)}
=\ol{\Gamma}\vrt_{zI}\,,
\quad z\in\bbC_*\,.
\]
In the terminology of the previous section,
$M_\Gamma(z)$ is the Shmul'yan transform of
$zI$ induced by $\Gamma$, and similarly for
$M_{\ol{\Gamma}}(z)$.
It follows that
\begin{equation}
\ol{\Gamma_z}\subseteq(\ol{\Gamma})_z\,,
\label{eq:Gzineq}
\end{equation}
that $(\ol{\Gamma})_z$ is a closed relation,
and that
\[
M_\Gamma(z)\subseteq M_{\ol{\Gamma}}(z)
\quad\text{and then}\quad
\ol{M_\Gamma(z)}\subseteq \ol{M_{\ol{\Gamma}}(z)}\,.
\]

An application of Lemma~\ref{lem:Torth} requires that
there be the equality in \eqref{eq:Gzineq}. No
additional efforts are needed if $\Gamma=\ol{\Gamma}$;
otherwise we apply Lemma~\ref{lem:clVTan0001}.
Assuming for a moment $\ol{\Gamma_z}=(\ol{\Gamma})_z$
it immediately follows that
\begin{equation}
\ol{\fN_z(A_*)}=\ol{\fN_z(A_\star)}\,,\quad
\ol{M_\Gamma(z)}=\ol{M_{\ol{\Gamma}}(z)}\,.
\label{eq:MMMGGGs}
\end{equation}
This shows
$M_\Gamma(z)^*=M_{\ol{\Gamma}}(z)^*$, so
it suffices to put in Lemma~\ref{lem:Torth}
$V\equiv\ol{\Gamma}$ instead of $\Gamma$.

Let
\[
M_{\Gamma_\#}(z):=\Gamma_\#(\whfN_z(A_\#))
=\ran(\Gamma_\#\vrt_{zI})\,,
\quad
A_\#:=\dom\Gamma_\#\,.
\]
Recall that $\Gamma_\#:=(\Gamma^+)^{-1}$,
where $\Gamma^+$ is the Krein space adjoint of
$\Gamma\co\whJ_\fH\lto\whJ^\circ_\fL$.

For later reference we observe that
\[
\ol{A_\#}=S^+\quad\text{where}\quad
S:=\ker\ol{\Gamma}
\]
and $S\subseteq T=\ker\Gamma_\#$.
\begin{thm}\label{thm:rrz}
Let $T\in\wtscrC_s(\fH)$,
let $(\fL,\Gamma)$ be an \ibp for $T^+$, and
let $M_{\ol{\Gamma}}$ be the Weyl family corresponding
to an \ibp $(\fL,\ol{\Gamma})$ for $T^+$.
Then
\begin{equation}
M_{\ol{\Gamma}}(z)^*=M_{\Gamma_\#}(\ol{z})
\label{eq:MGG}
\end{equation}
for all $z\in\bbC_*$ such that $\ran(A_\star-zI)$ is
a subspace;
if also $\ran(A_\#-\ol{z}I)$ is a subspace, then
$M_{\ol{\Gamma}}(z)\in\wtscrC(\fL)$.
\end{thm}
\begin{proof}
We apply Lemma~\ref{lem:Torth} to
$M_{\ol{\Gamma}}(z)=\ol{\Gamma}(zI)$.
Condition \ref{item:V} therefore requires that
$zI\hsum A_\star$ be a closed relation:
The adjoint of
\[
zI\hsum A_\star=zI\hsum
(\{0\}\times\ran(A_\star-zI))
\]
is given by
\[
(zI\hsum A_\star)^+=\ol{z}I\mcap T=
\whfN_{\ol{z}}(T)
\]
and then the closure
\[
\ol{zI\hsum A_\star}=\whfN_{\ol{z}}(T)^+
=zI\hsum
(\{0\}\times\fN_{\ol{z}}(T)^{[\bot]})
\]
where the $J$-orthogonal complement
\[
\fN_{\ol{z}}(T)^{[\bot]}=\ol{\ran}(T^+-zI)=
\ol{\ran}(A_\star-zI)\,.
\]
By Lemma~\ref{lem:Pop} therefore,
$zI\hsum A_\star\in\wtscrC(\fH)$ iff
$\ran(A_\star-zI)$ is closed.

Similarly, condition \ref{item:Vp} requires
$\ol{z}I\hsum A_\#$ to be a closed relation,
which happens iff $\ran(A_\#-\ol{z}I)$ is closed.
\end{proof}
Stated otherwise:

$\circ$
if $\ran(A_\star-zI)$ is a
subspace, then there is the equality in
the inclusion
\[
M_{\ol{\Gamma}}(z)^*=\mul(\ol{\Gamma_\#\hsum(\ol{z}I
\times\{0\}) })\supseteq M_{\Gamma_\#}(\ol{z})\,,
\quad z\in\bbC_*\,.
\]

$\circ$
if $\ran(A_\#-\ol{z}I)$ is a
subspace, then there is the equality in
the inclusion
\[
M_{\Gamma_\#}(\ol{z})^*=\mul(\ol{\ol{\Gamma}\hsum(zI
\times\{0\}) })\supseteq M_{\ol{\Gamma}}(z)\,,
\quad z\in\bbC_*\,.
\]

According to the conclusion of Theorem~\ref{thm:rrz}
it is necessary that $\ran(T^+-zI)$ be a subspace.
If also $\ran(S^+-\ol{z}I)$ is a subspace,
we denote by $\Omega_\Gamma$ the set of all such $z\in\bbC_*$:
\begin{equation}
\Omega_\Gamma:=
\{z\in\bbC_*\vrt\,
\ran(T-\ol{z}I)\;\text{and}\;
\ran(S-zI)\;\text{are subspaces} \}
\,.
\label{eq:OmG}
\end{equation}
For example, if
$\fH$ is a Pontryagin space then
$\Omega_\Gamma=\bbC_*$ \cite{Azizov03}.

The next proposition gives equivalent
conditions for the closeness of
$\ran(A_\star-zI)$ and $\ran(A_\#-\ol{z}I)$
in Theorem~\ref{thm:rrz}.
\begin{prop}\label{prop:rrz}
Assume the conditions of Theorem~\ref{thm:rrz}
and let $z\in \Omega_\Gamma\neq\emptyset$.
\begin{enumerate}[label=\arabic*$^\circ$,
ref=\arabic*$^\circ$]
\item
$\ran(A_\star-zI)$ is a subspace iff
there is a $\wtT\in\scrC(\fH)$,
$\wtT\subseteq A_\star$, such that
\begin{equation}
T^+=\wtT\hsum\whfN_z(T^+)\,.
\label{eq:TpfN}
\end{equation}
\item
$\ran(A_\#-\ol{z}I)$ is a subspace iff
there is a $\wtT\in\scrC(\fH)$,
$\wtT\subseteq A_\#$, such that
\begin{equation}
S^+=\wtT\hsum\whfN_{\ol{z}}(S^+)\,.
\label{eq:TpfN2}
\end{equation}
\end{enumerate}
\end{prop}
The proposition is an immediate
corollary of the next lemma, \eg
\cite[Theorem~1.7.1]{Behrndt20},
\cite[Lemma~1.4]{Hassi07b},
\cite[Corollary~A.4]{Wietsma12}.
\begin{lem}\label{lem:Behrndt20}
Let $\fH$ be a linear set, let $L$, $R\in\scrC(\fH)$,
$L\subseteq R$, and $z\in\bbC$.
Equivalent are:
\begin{enumerate}[label=(\roman*),
ref=\roman*]
\item\label{item:Behrndt20-i}
$R=L\hsum\whfN_z(R)$.
\item\label{item:Behrndt20-ii}
$\ran(R-zI)=\ran(L-zI)$.
\end{enumerate}
\end{lem}
Note that $L$ is not unique in that,
if at least one of
\eqref{item:Behrndt20-i}--\eqref{item:Behrndt20-ii}
holds then
$\ran(R-zI)=\ran(L_1-zI)$ for
every $L_1\in\scrC(\fH)$ such that
$L\subseteq L_1\subseteq R$; hence
\eqref{item:Behrndt20-i}
holds for each such $L_1$ in place of $L$.
For later reference we mention that
\eqref{item:Behrndt20-i} holds for each such $L_1$
in place of $R$.
\begin{cor}\label{cor:projP1}
Let $L_1\in\scrC(\fH)$ such that
$L\subseteq L_1\subseteq R$.
If at least one of \eqref{item:Behrndt20-i}
or \eqref{item:Behrndt20-ii} holds then also
\[
L_1=L\hsum\whfN_z(L_1)
\]
holds.
\end{cor}
A familiar example in a Hilbert space setting,
to which Theorem~\ref{thm:rrz}
and Proposition~\ref{prop:rrz} apply,
is a $B$-generalized \bp\!\!; that is,
a \ubp $(\fL,\Gamma)$ for $T^*$ such that
$\ran\Gamma_0=\fL$, and then such that
$\wtT=T_0$ is self-adjoint in a Hilbert space $\fH$;
see \cite[Definition~5.6, Proposition~5.9]{Derkach06},
\cite[Definition~3.5]{Derkach17}. In this
case $S=T$ and \eqref{eq:TpfN}
(or \eqref{eq:TpfN2}) is known as
the von Neumann formula.

That the closeness of $\ran\Gamma_0$ implies the
self-adjointness of $T_0$ is shown in
\cite[Corollary~4.17]{Derkach06}.
In a Krein space setting this
result follows from Proposition~\ref{prop:equivfNTh}.
\begin{prop}\label{prop:equivfNTh}
Let $T\in\wtscrC_s(\fH)$, let
$(\fL,\Gamma)$ be an \eubp for $T^+$.
Given an arbitrary $\Theta\in\scrC(\fL)$,
let
\begin{equation}
\mathring{T}_\Theta:=\ol{\Gamma}^{\;-1}(\Theta)
\quad\text{so that}\quad
\mathring{T}_\Theta\in\mrm{Ext}_T\,,\quad
\mathring{T}_\Theta\subseteq A_\star
\label{eq:inmrTTh}
\end{equation}
($\mathring{T}_\Theta=S=T$ for
$\Theta\mcap\ran\ol{\Gamma}\subseteq\mul\ol{\Gamma}$;
$\mathring{T}_\Theta=A_\star$ for
$\Theta\supseteq\ran\ol{\Gamma}$).

Suppose that
\[
\Theta\in\wtscrC(\fL)\quad\text{and}\quad
\Theta\hsum\ran\ol{\Gamma}\in\wtscrC(\fL)
\quad\text{and}\quad
\Theta^*\hsum\ran\ol{\Gamma}\in\wtscrC(\fL)\,.
\]
Then $\mathring{T}_\Theta\in\wtscrC(\fH)$ whose
adjoint
$\mathring{T}^+_\Theta=\mathring{T}_{\Theta^*}$.
\end{prop}
When $J=I$ and $\Gamma=\Gamma_\#$
is a (reduction) operator,
the existence of a maximal symmetric
$\wtT\in\mrm{Ext}_T$ such that
$\wtT^*\subseteq A_*$ ($=A_\star$)
is stated in \cite[Corollary~2.12]{Derkach12}.
\begin{proof}
The statement in \eqref{eq:inmrTTh} is clear from
\[
\ol{\Gamma}^{\;-1}(\Theta)=
\ol{\Gamma}^{\;-1}(\Theta\mcap\ran\ol{\Gamma})\,.
\]
Since
$\ol{\Gamma}^{\;-1}\vrt_\Theta$ is closed for
a closed $\Theta$,
the adjoint
of $\mathring{T}_\Theta$ follows from
Lemma~\ref{lem:Torth}-\ref{item:Torth-1o}.
That $\mathring{T}_\Theta$ is closed is due to
Lemma~\ref{lem:Torth}-\ref{item:Torth-2o}.
\end{proof}
Letting $\Theta=\{0\}\times\fL$ and
using $\dom\ran\ol{\Gamma}=\ran(\ol{\Gamma})_0$
it follows that, if $\ran(\ol{\Gamma})_0$ is closed,
then $\mathring{T}_0$ is self-adjoint.
In a Hilbert space case,
a \ubp $(\fL,\Gamma)$ for $T^*$
such that $T_0$ is self-adjoint
is an $S$-generalized \bp
(\cite[Definition~5.11]{Derkach17}). Thus,
every \ubp $(\bbC^d,\Gamma)$ for $T^*$, with $d\in\bbN$,
is an $S$-generalized \bp
\begin{rem}
If $\ran\ol{\Gamma}$ is closed, then
$A_\star=T^+$ and one extracts from
Proposition~\ref{prop:equivfNTh}
a 1-1 correspondence between all $\Theta\in\wtscrC(\fH)$
such that $\mul\ol{\Gamma}\subseteq\Theta
\subseteq\ran\ol{\Gamma}$
($=\ol{\ran}\Gamma$)
and all closed $\mathring{T}_\Theta\in\mrm{Ext}_T$.
A similar correspondence principal in a Hilbert
space case is stated in
\cite[Theorem~7.10]{Derkach12} (see also reference therein)
in case $\ol{\Gamma}$ is an operator.
\end{rem}
We write down several corollaries of
Theorem~\ref{thm:rrz}.
In what follows, given a subset $\cO\subseteq\bbC$,
we let $\cO^*:=\{z\vrt\ol{z}\in\cO\}$, so that
the subsets
$\cO\mcup\cO^*$ and $\cO\mcap\cO^*$, if nonempty,
are symmetric with respect to the real axis.

As an example,
if $\Pi_\Gamma$ is an \obt for $T^+$, then for
$z\in\Omega_\Gamma$,
$M_\Gamma(z)\in\wtscrC(\fL)$ and
$M_\Gamma(z)^*=M_\Gamma(\ol{z})$.
In this case $\Omega_\Gamma$, if nonempty, is symmetric
(with respect to $\bbR$); see \eqref{eq:OmG}.
Moreover $\Omega_\Gamma\mcap\res T_0=
\bbC_*\mcap\res T_0 $.
Thus, if $\bbC_*\mcap\res T_0\neq\emptyset$
then $\Omega_\Gamma\neq\emptyset$, but the
converse implication does not necessarily hold;
\ie $M_\Gamma(z)$ need not be an operator
for $z\in\Omega_\Gamma$.

The next corollary is useful
for the characterization of $M_{\ol{\Gamma}}(z)^*$
on a symmetric subset of $\Omega_\Gamma\mcap\Omega^*_\Gamma$.
\begin{cor}\label{cor:rrzz}
Assuming the hypotheses of Theorem~\ref{thm:rrz}
suppose in addition that the set
\[
\delta_\Gamma:=
(\Omega_\Gamma\setm\sigma^0_p(T))\mcap
(\Omega_\Gamma\setm\sigma^0_p(T))^*
\]
is nonempty and let $z\in\delta_\Gamma$.
Then:
\begin{enumerate}[label=\arabic*$^\circ$,
ref=\arabic*$^\circ$]
\item\label{item:rrzz-1o}
$\ran(A_*-z I)=\fH$.
\item\label{item:rrzz-2o}
$M_{\ol{\Gamma}}(z)^*=M_{\Gamma_\#}(\ol{z})$ and
$M_{\ol{\Gamma}}(z)\in\wtscrC(\fL)$.
\end{enumerate}
\end{cor}
\begin{proof}
\ref{item:rrzz-1o}
By contradiction, suppose
$(\exists z\in\delta_\Gamma)$
$\ran(A_*-z I)\subsetneq\fH$. Since
the adjoint $A^+_*=T$, this implies
$(\exists z\in\delta_\Gamma)$
$\fN_{\ol{z}}(T)\neq\{0\}$, which is false
by definition.

\ref{item:rrzz-2o} Using
$A_*\subseteq A_\star\subseteq A_\#$,
this follows from
\ref{item:rrzz-1o} and Theorem~\ref{thm:rrz}.
\end{proof}
The set $\delta_\Gamma$ admits a representation
\begin{equation}
\delta_\Gamma=\delta_T\mcap\delta_S\quad\text{where}
\quad
\delta_T:=(\bbC_*\mcap\reg T)\mcap(\bbC_*\mcap\reg T)^*
\label{eq:dT}
\end{equation}
and similarly for $\delta_S$.

Recall that, if $\fH$ is a Pontryagin space
and $T\in\wtscrC_s(\fH)$ then
$\Omega_\Gamma=\bbC_*$, and then $\delta_\Gamma$
becomes equal to
\[
\delta_T=(\bbC_*\setm\sigma^0_p(T))\mcap
(\bbC_*\setm\sigma^0_p(T))^*\,.
\]
In this setting a sufficient condition for
$\delta_T\neq\emptyset$ is given next.

The set $\sigma^0_p(T^+)$ contains the
subset (\cf \cite[Section~2.6]{Azizov79} for notation)
\[
\sigma^0_{p,2}(T^+):=\{z\in\sigma^0_p(T^+)\vrt
\ran(T^+-zI)=\fH \}\,.
\]
Letting $z\in\bbC_*\setm\sigma^0_p(T)$ we have
$\ran(T^+-\ol{z}I)=\fH$, so that
either $\ol{z}\in\sigma^0_{p,2}(T^+)$ or
$\ol{z}\in\bbC_*\mcap\res T^+$.
\begin{prop}
Let $\fH$ be a Pontryagin space
and let $T\in\wtscrC_s(\fH)$ be such that
$\bbC_*\setm\sigma^0_p(T)\neq\emptyset$.
Suppose that at least one of the two sets, either
\begin{enumerate}[label=$($\alph*$)$,
ref=\alph*]
\item\label{item:Pon-a}
$\sigma^0_{p,2}(T^+)\setm\sigma^0_p(T)$ or
\item\label{item:Pon-b}
$\bbC_*\mcap\res T^+$,
\end{enumerate}
is nonempty. Then the set
$\delta_T$ is also nonempty.
\end{prop}
\begin{proof}
Let $\sigma^0_{p,1}(T^+)$ be the complement in
$\sigma^0_p(T^+)$ of $\sigma^0_{p,2}(T^+)$ and
let $\sigma^0_r(T^+)$ denote the part of
the residual spectrum of $T^+$ in $\bbC_*$. That is
\begin{align*}
\sigma^0_{p,1}(T^+):=&
\{z\in\sigma^0_p(T^+)\vrt 
\ran(T^+-zI)\subsetneq\fH \}\,,
\\
\sigma^0_r(T^+):=&
\{z\in\bbC_*\setm\sigma^0_p(T^+)\vrt
\ran(T^+-zI)\subsetneq\fH \}\,.
\end{align*}
The sets $\sigma^0_{p,1}(T^+)$,
$\sigma^0_{p,2}(T^+)$, $\sigma^0_r(T^+)$
are mutually disjoint, whose union is the total
spectrum of $T^+$ contained in $\bbC_*$.

Let $z\in\bbC_*\setm\sigma^0_p(T)$ such that
$\ol{z}\in\sigma^0_p(T)$. Then either
$z\in\Sigma_1:=\sigma^0_{p,1}(T^+)\setm\sigma^0_p(T)$
or $z\in\Sigma_2:=\sigma^0_r(T^+)\setm\sigma^0_p(T)$.
Thus if we show that the set
\[
(\bbC_*\setm\sigma^0_p(T))\setm\Sigma_{12}
\quad\text{where}\quad
\Sigma_{12}:=\Sigma_1\mcup\Sigma_2
\]
is nonempty, then by taking $z$ from this set
we must have $z\in\delta_T$.
By noting that
$(\bbC_*\setm\sigma^0_p(T))\setm\Sigma_{12}$
is the union of sets in
\eqref{item:Pon-a} and \eqref{item:Pon-b},
this accomplishes the proof of the proposition.
\end{proof}
The example with $\delta_T\neq\emptyset$
is a standard relation in a Pontryagin space.
\begin{exam}\label{exam:Pstan}
Let $\fH$ be a Pontryagin space with $\kappa$
negative squares and $T\in\wtscrC_s(\fH)$.
Then $T$ is called \textit{standard} if $(\exists z\in\bbC_*)$
$\ol{z}$, $z\notin\sigma^0_p(T)$, in which case
$\sigma^0_p(T)$ is at most $2\kappa$-dimensional
\cite{Azizov03}. Thus, a standard relation $T$
is the one for which the set $\delta_T$ is nonempty.
Recall also that $T\in\wtscrC_s(\fH)$ is called
\textit{simple}, if it has no non-real eigenvalues and
the closed linear span
\[
\ol{\spn}\{\fN_z(T^+)\vrt z\in\bbC_*\}=\fH\,.
\]
A simple $T\in\wtscrC_s(\fH)$ is a
standard operator with $\sigma_p(T)=\emptyset$
\cite[Proposition~2.4]{Azizov03}.
Therefore, $\delta_T=\bbC_*$ for a simple operator $T$.

Let
\[
\cN:=\mul T^+\mcap\spn
\{\ker(T-zI)^\kappa\vrt z\in\sigma_p(T) \}
\]
so that $\cN=\{0\}$ if \eg
$\kappa=0$ or $T$ is simple or $T$ is densely defined.
According to
\cite[Theorem~3.7]{Azizov03}, if $T$ is a standard operator
then there is a closed $\wtT\in\mrm{Ext}_T$
such that \eqref{eq:TpfN} holds for $z\in\delta_T$
(and is termed a \textit{generalized von Neumann formula})
iff $\cN=\{0\}$; see also
\cite[Lemmas~3.3(b), 3.6]{Azizov03}.
For a standard operator $T\subsetneq T^+$ with $\cN=\{0\}$,
the set in \eqref{item:Pon-a} is nonempty.
\end{exam}
The last corollary of Theorem~\ref{thm:rrz}
that we mention explicitly in this section assumes $S=T$
(\ie $\Omega_\Gamma$ is symmetric)
and that the defect subspaces of $T$
are finite-dimensional. We remark that
$S=T$ does not necessarily imply that
$\ol{\Gamma}$ is unitary (but it is, if \eg
$\ran\Gamma_\#=\ran\ol{\Gamma}$ too). Let
\[
n_z:=\dim\fN_{\ol{z}}(T^*)
=\dim\fN_{\ol{z}}(T^+)
\]
for some $z\in\bbC_*$. We do not require
$n_z$ to be constant, but only that $n_z<\infty$
for some $z$. Recall in this connection that, if
$\reg T\neq\emptyset$ then for every
$z$ from each connected component of $\reg T$,
$n_z$ is constant \cite{Jonas95,Gokhberg57}.
For example, if $\fH$ is a Pontryagin space,
if $T$ is standard,
and if $n_z$, $n_{\ol{z}}<\infty$ for
$z\in\bbC_+\mcap\delta_T$,
then $n_z$ is constant
for all $z\in\bbC_\pm$, with the exception of at most
finitely many points
\cite[Theorem~2.3]{Azizov03}.
\begin{cor}\label{cor:rrz}
Assuming the hypotheses of Theorem~\ref{thm:rrz}
suppose in addition that
$S=T$ and $n_z$, $n_{\ol{z}}<\infty$ for
$z\in \Omega_\Gamma\neq\emptyset$.
Then \eqref{eq:MGG} holds and
$M_{\ol{\Gamma}}(z)\in\wtscrC(\fL)$.
\end{cor}
\begin{proof}
Fix some $z\in\Omega_\Gamma$ and let
\[
\fH_z:=\ran(T-z I)\,,\quad
\fX_z:=\ran(A_\star-z I)\,,\quad
\fH^\prime_z:=\ran(T^+-z I)\,.
\]
By hypotheses $\fH_z$ and $\fH^\prime_z$
are Hilbert subspaces of $(\fH,[\cdot,J\cdot])$.
The closeness of $\fX_z$
is equivalent to the closeness of
\[
\hat{\fX}_z:=
\fX_z\mcap(\fH^\prime_z\om\fH_z)\,.
\]
Indeed,
since $\fH_z\subseteq\fX_z$,
every element from $\fX_z$ is the unique sum
of an element from $\fH_z$ and an element
from $\hat{\fX}_z$.
Since $\fH_z$ is closed, the closeness
of $\hat{\fX}_z$ therefore implies that of $\fX_z$.
The converse implication is clear.
Since $n_z<\infty$, the subspace
\[
\fH^\prime_z\om\fH_z=
\fH^\prime_z\mcap(\fH\om\fH_z)=
\fH^\prime_z\mcap\fN_{\ol{z}}(T^*)
\]
is finite-dimensional, hence closed, and
then $\fX_z$ is also closed.
The proof of the closeness of
$\ran(A_\#-\ol{z} I)$ is similar.
\end{proof}
Next we give conditions such that
\eqref{eq:MMMGGGs} holds.
Specifically,
Lemmas~\ref{lem:clVTan0001}, \ref{lem:Behrndt20}
and Corollary~\ref{cor:projP1} lead to Theorem~\ref{thm:PzfN}.
\begin{thm}\label{thm:PzfN}
Let $T\in\wtscrC_s(\fH)$,
let $(\fL,\Gamma)$ be an \ibp for $T^+$
and suppose \eqref{eq:TpfN} holds for
some $\wtT\in\scrC(\fH)$, $\wtT\subseteq T^+$,
and $z\in\bbC$.
\begin{enumerate}[label=\arabic*$^\circ$,
ref=\arabic*$^\circ$]
\item\label{item:PzfN-1o}
If $\wtT\subseteq A_\star$ and $z\neq0$ and
\begin{equation}
\fN_z(T^+)\mcap\ran(A_*+\frac{1}{\ol{z}} I)=\fN_z(A_*)
\label{eq:cAT}
\end{equation}
then $\ol{\fN_z(A_*)}=\ol{\fN_z(A_\star)}$;
if additionally $\Gamma$ is a closable operator
with $A_\star=T^+$ then
there is the equality in \eqref{eq:Gzineq}.
\item\label{item:PzfN-2o}
If $z\in\bbC\setm\sigma_p(\ol{\wtT})\neq\emptyset$
then for every $\wtT_1\in\scrC(\fH)$ such that
$\wtT\subseteq\wtT_1$ and $\ol{\wtT_1}=T^+$
it holds $\ol{\fN_z(\wtT_1)}=\fN_z(T^+)$.
\end{enumerate}
\end{thm}
\begin{proof}
\ref{item:PzfN-1o}
This is an application of Lemma~\ref{lem:clVTan0001}
by noting that \eqref{eq:TpfN} holds with $\wtT=A_\star$
by Lemma~\ref{lem:Behrndt20}, and that
$A_\star\hsum\whfN_z(T^+)$ is closed
(condition \ref{item:V}) iff it equals
$T^+$.

Let $P_z$ be an orthogonal projection
in $\fH^2$ onto $\whfN_z(T^+)$. We
show $P_zA_*\subseteq A_*$ for $z\in\bbC\setm\{0\}$
iff \eqref{eq:cAT} holds.
Conditions $P^*_z=P_z=P^2_z$ imply that
$P_z$ has the matrix form
\[
P_z=\begin{pmatrix}
\displaystyle\frac{\pi_z}{\abs{z}^2+1}
& \displaystyle\frac{\ol{z}\pi_z}{\abs{z}^2+1}
\\
\displaystyle\frac{z\pi_z}{\abs{z}^2+1}
& \displaystyle\frac{\abs{z}^2\pi_z}{\abs{z}^2+1}
\end{pmatrix}
\]
where $\pi_z$ is an orthogonal projection
in $\fH$ onto $\fN_z(T^+)$.
Therefore $P_zA_*\subseteq A_*$ iff
\[
\pi_z\ran(A_*-\mu I)\subseteq
\fN_z(A_*)\quad\text{where}\quad
\mu:=-\frac{1}{\ol{z}}\,,\quad
z\neq0
\]
or, equivalently (see \eqref{eq:nftr}), iff
\[
\fN_z(T^+)\mcap\ran(A_*-\mu I)=
\fN_z(A_*)\mcap\ran(A_*-\mu I)\,.
\]
The set on the right
consists of $f^\prime-\mu f$ such that
\[
(f,f^\prime)\in A_*\,,\quad
(f^\prime-\mu f,zf^\prime-\mu zf)\in A_*\,.
\]
That is
\[
(\forall\lambda\neq\mu)\;
\fN_z(A_*)\mcap\ran(A_*-\mu I)\supseteq
\fN_\lambda(A_*)\mcap\fN_z(A_*)\,.
\]
By noting that $z\neq\mu$ and putting
$\lambda=z$ it follows that
\[
\fN_z(A_*)\mcap\ran(A_*-\mu I)=\fN_z(A_*)\,.
\]

\ref{item:PzfN-2o}
By Lemma~\ref{lem:Behrndt20} and Corollary~\ref{cor:projP1}
\[
T^+=\ol{\wtT}\hsum\whfN_z(T^+)\,,\quad
\wtT_1=\wtT\hsum\whfN_z(\wtT_1)\,.
\]
Since
$\ol{\wtT}\hsum\whfN_z(T^+)$ is closed,
the angle, say $\theta$, between subspaces
$\ol{\wtT}$ and $\whfN_z(T^+)$ is less than $1$
(\eg \cite[Theorem~2.1]{Schochetman01}).
Since
\[
\ol{\wtT}\mcap\whfN_z(T^+)=\whfN_z(\ol{\wtT})=\{0\}
\]
the angle between subspaces $\ol{\wtT}$ and
$\ol{\whfN_z(\wtT_1)}\subseteq\whfN_z(T^+)$ is
less than or equal to $\theta<1$, and then the
relation
$\ol{\wtT}\hsum\ol{\whfN_z(\wtT_1)}$ is also closed.
Then
\[
\ol{\wtT}\hsum\ol{\whfN_z(\wtT_1)}=\ol{\wtT_1}=T^+=
\ol{\wtT}\hsum\whfN_z(T^+)
\]
which shows that $\ol{\whfN_z(\wtT_1)}=\whfN_z(T^+)$.
\end{proof}
The example in a Hilbert space setting, where
$\Pi_\Gamma$ is an \ibt for $T^*$ such that
$\Pi_{\ol{\Gamma}}$ is an \obt for $T^*$
and $\fN_z(A_*)=\fN_z(T^*)$,
can be found in \cite{Jursenas18}. But
observe that generally
an \ibt $\Pi_{\ol{\Gamma}}$ for $T^+$ with
$A_\star=T^+$ need not be an \obt
If $(\fL,\Gamma)$ is an \ibp for $T^+$
and if $\fN_z(A_*)=\fN_z(A_\star)$, then
by Lemma~\ref{lem:Pop}
there is the equality in \eqref{eq:Gzineq} iff
$\mul\ol{\Gamma_z}=\mul\ol{\Gamma}$.

Under suitable circumstances a necessary
and sufficient condition for the equality
$\fN_z(A_*)=\fN_z(T^+)$ can be extracted
from the following proposition.
\begin{prop}
Let $L\subseteq R$ be relations in a linear set $\fH$.
Suppose that
\[
\mul R\subseteq\ran L\,,\quad
\dom(L\mcap\fX)=\ker L
\]
where
\[
\fX:=(\dom L\mcap\ker R)\times\mul R\,.
\]
Then
\[
\fN_z(R)\mcap\dom L=\fN_z(L)\,,\quad z\in\bbC\,.
\]
\end{prop}
\begin{proof}
Let $f\in\fN_z(R)\mcap\dom L$.
We have
\[
\fN_z(R)\mcap\dom L=\fN_z(R\vrt_{\dom L})\,,\quad
R\vrt_{\dom L}=L\hsum(\{0\}\times\mul R)
\]
so $(f,zf+h)\in L$ for some $h\in\mul R$. By
hypothesis $\mul R\subseteq\ran L$,
$(\exists g)$ $(g,h)\in L$ so
\[
(f,zf+h)=(f,zf)+(g,h)-(g,0)
\]
and then
\[
(f,zf)-(g,0)\in L\,.
\]
By hypothesis $\dom(L\mcap\fX)=\ker L$ or,
equivalently,
\[
\dom(L\mcap(\fH\times\mul R))\subseteq\ker L
\]
this gives $f\in\fN_z(L)$.
\end{proof}
Let $\fM$, $\fN$ be linear subsets of
$\fH$. The following example illustrates that
a combination of Lemma~\ref{lem:clVTan0001}
with \eqref{eq:cAT} can be useful for establishing
the equality in the inclusion
$\ol{\fM\mcap\fN}\subseteq\ol{\fM}\mcap\ol{\fN}$.
\begin{exam}
Let $(\fH,[\cdot,\cdot])$ be a Krein space
and let $\fL=(\fH,[\cdot,J\cdot])$, a Hilbert space.
Let $\fM$, $\fN$ be linear subsets of $\fH$.
\begin{enumerate}[label=\arabic*$^\circ$,
ref=\arabic*$^\circ$]
\item\label{item:ex-1o}
If $(I-J)\fH\subseteq\fN^{[\bot]}\subseteq\ol{\fM}$
then the pair $(\fL,\Gamma)$, with
\[
\Gamma:=\{((f,f^\prime),(f+l,f^\prime))\vrt
(f,f^\prime)\in\fM\times\fN\,;\,
l\in\fN^{[\bot]} \}\,,
\]
is an \ibp for the adjoint
$T^+=\ol{\fM}\times\ol{\fN}$ of
$T=\fN^{[\bot]}\times\fM^{[\bot]}\in\wtscrC_s(\fH)$.
It holds $A_*=\fM\times\fN$, $A_\star=T^+$, and
moreover
\[
\fN_z(A_*)=\fM\mcap\fN\,,\quad
\fN_z(T^+)=\ol{\fM}\mcap\ol{\fN}
\]
for all $z\in\bbC\setm\{0\}$.
\item\label{item:ex-2o}
If
$(\fM\mcap\ol{\fN})+(\ol{\fM}\mcap\fN)=\fM\mcap\fN$
then
$\ol{\fM\mcap\fN}=\ol{\fM}\mcap\ol{\fN}$.
\item\label{item:ex-3o}
Under hypotheses of \ref{item:ex-1o},
\ref{item:ex-2o}, if in addition
$\fN$ is dense in $\fH$, then necessarily $J=I$ and
moreover
there is the equality in \eqref{eq:Gzineq}
for $z\in\bbC\setm\{0\}$.
\end{enumerate}
\begin{proof}
\ref{item:ex-1o}
By the definition of $\Gamma$ it holds
\[
\Gamma_\#=
\{((Jf+l,f^\prime+l^\prime),(f,Jf^\prime))\vrt
(f,f^\prime)\in\fH\times\ol{\fN}\,;\,
(l,l^\prime)\in\fN^{[\bot]}\times\fM^{[\bot]} \}
\]
and
\[
\ol{\Gamma}=
\{((f,f^\prime),(f+l,f^\prime))\vrt
(f,f^\prime)\in\ol{\fM}\times\ol{\fN}\,;\,
l\in\fN^{[\bot]} \}\,.
\]

By hypothesis $(I-J)\fH\subseteq\fN^{[\bot]}$,
if $f\in\fH$ and $l\in\fN^{[\bot]}$
then
\[
Jf+l=f+l_1\,,\quad
l_1:=l+(J-I)f\in\fN^{[\bot]}\,.
\]
Moreover, if $f^\prime\in\ol{\fN}$
($\subseteq(I+J)\fH$) then
$Jf^\prime=f^\prime$. Thus $\Gamma_\#$ reduces to
\[
\Gamma_\#=
\{((f,f^\prime),(f+l,f^\prime+l^\prime))\vrt
(f,f^\prime)\in\fH\times\ol{\fN}\,;\,
(l,l^\prime)\in\fN^{[\bot]}\times\fM^{[\bot]} \}
\]
and then $\Gamma_\#\supseteq\Gamma$.

By hypothesis $\fN^{[\bot]}\subseteq\ol{\fM}$
or, equivalently, $\fM^{[\bot]}\subseteq\ol{\fN}$,
we have $T\in\wtscrC_s(\fH)$. The remaining statements
are straightforward.

\ref{item:ex-2o}
This is an application of
Lemma~\ref{lem:clVTan0001} and \eqref{eq:cAT},
since
\[
\ran(A_*-zI)=\fM+\fN
\]
for all $z\in\bbC\setm\{0\}$.

\ref{item:ex-3o}
Since $\fN^{[\bot]}=\{0\}$, it is necessary
by \ref{item:ex-1o}
that $J=I$. Moreover $\Gamma$ is a closable
operator with $A_\star=T^+=T^*$, since
$\mul\ol{\Gamma}=\fN^{[\bot]}$; hence the
last statement of
Theorem~\ref{thm:PzfN}-\ref{item:PzfN-1o} applies.
\end{proof}
\end{exam}
\begin{rem}
We remark that in the above example
the hypothesis in \ref{item:ex-1o} is not necessary
for establishing $\ol{\Gamma_z}=(\ol{\Gamma})_z$.
Namely, one verifies
$\ol{\Gamma_z}=(\ol{\Gamma})_z$ by
direct computation, by using
$\Gamma_z=I_{zI_{\fM\mcap\fN}}$
and $(\ol{\Gamma})_z=I_{zI_{\ol{\fM}\mcap\ol{\fN}}}$,
and then by applying \ref{item:ex-2o}.
But then $(\fL,\Gamma)$ generally would not be a \bp
\end{rem}
We give a comment related to
Theorem~\ref{thm:PzfN}-\ref{item:PzfN-2o}
in case $\wtT_1=A_*$. For this, we first recall that,
if $(\fL,\Gamma)$ is a
\bp then the main transform, $\cJ(\Gamma)$,
of $\Gamma$ is a relation in $\fH\times\fL$
defined by
(\cite[Eq.~(2.16)]{Derkach06},
\cite[Eq.~(3.5)]{Behrndt11})
\begin{equation}
\cJ(\Gamma):=
\{((f,l),(f^\prime,-l^\prime))\vrt
(\whf,\whl)\in\Gamma \}\,,\quad
\whf=(f,f^\prime)\,,\quad\whl=(l,l^\prime)\,.
\label{eq:mrTG}
\end{equation}
If $\Gamma$ is a \br for $T^+$ then $\cJ(\Gamma)$ is a
self-adjoint extension to $\fH\times\fL$ of
$T=\ker\Gamma$, in the sense that
\[
\cJ(\Gamma)\mcap(\fH\times\{0\})^2=
\{((f,0),(f^\prime,0))\vrt\whf\in T \}\,.
\]
This shows in particular that
$\sigma_p(T)\subseteq\sigma_p(\cJ(\Gamma))$,
and therefore in order to have
$\res\cJ(\Gamma)\neq\emptyset$ it is necessary
(although generally not sufficient) that
$\sigma_p(T)\subsetneq\bbC$.
Note also that
$\bbC_*\mcap\res\cJ(\Gamma)\subseteq\Omega_\Gamma$.
We examine $\res\cJ(\Gamma)$ in more detail
in the next section.

In a Hilbert space setting
the equality in
\begin{equation}
\ol{\fN_z(A_*)}\subseteq\fN_z(T^+)
\label{eq:fNzineqq}
\end{equation}
for $z\in\bbC_*$ is shown in
\cite[Proposition~3.9(i)]{Derkach17},
provided $T_0$ is a self-adjoint relation,
and is therefore a consequence of the von Neumann formula
(put $\wtT=T_0$ and $\wtT_1=A_*$
in Theorem~\ref{thm:PzfN}-\ref{item:PzfN-2o}).

Let $\fH$ be a Pontryagin space, $T\in\wtscrC_s(\fH)$
an operator, and $\Gamma$ a \br
for $T^+$. Under the present circumstances
the equality in \eqref{eq:fNzineqq}
for some $z\in\bbC$ is stated in
\cite[Corollary~3.14(i)]{Behrndt11}
by using the idea from \cite[Lemma~2.14(ii)]{Derkach06}.
That is, the equality holds for
$z\in\res\cJ(\Gamma^\prime)$, where
$\Gamma^\prime$ is a \br for $T^+$ such that
$\dom\Gamma^\prime=A_*$ and
$\res\cJ(\Gamma^\prime)\neq\emptyset$.
In Section~\ref{sec:resTG} we find $\Gamma^\prime$
with $\res\cJ(\Gamma^\prime)$
intersecting $\delta_\Gamma=\delta_T$
(see Corollary~\ref{cor:rrzz} and \eqref{eq:dT}),
so that there is
the equality in \eqref{eq:fNzineqq}
for some $z\in\delta_T$.
\section{On the resolvent set of the
main transform of a boundary relation}\label{sec:resTG}
As in the previous section, $\fH$
is a Krein space, $\fL$ a Hilbert space,
$T\in\wtscrC_s(\fH)$. Let $(\fL,\Gamma)$
be a \ubp for $T^+$ with Weyl family $M_\Gamma$.
Then $\Omega_\Gamma$
consists of $z\in\bbC_*$ such that
$\ran(T-zI)$ and $\ran(T-\ol{z}I)$ are
subspaces. Relying on Theorem~\ref{thm:rrz},
in order that $M_\Gamma(z)$ be a closed relation
one would look for $z$ from $\Omega_\Gamma$. Let
\[
\cO_\Gamma:=\{z\in\Omega_\Gamma
\setm\sigma^0_p(T) \vrt
\ran(A_*-zI)=\fH \}
\]
and let $\delta_T$ be as defined in \eqref{eq:dT}.
Then by Corollary~\ref{cor:rrzz}
$\cO_\Gamma\mcap\cO^*_\Gamma=\delta_T$.

Let $\cJ(\Gamma)$ be the main transform of a \br $\Gamma$;
see \eqref{eq:mrTG}.
One verifies straightforwardly that
$\bbC_*\mcap\res\cJ(\Gamma)\subseteq\delta_T$.
Assuming
$\delta_T\neq\emptyset$
we show a kind of the converse inclusion,
that is, $\Gamma$ can be chosen such that
$\res\cJ(\Gamma)\neq\emptyset$.

Define the subset $\Sigma_\Gamma\subseteq\cO_\Gamma$ by
\[
\Sigma_\Gamma:=\{z\in\cO_\Gamma \vrt
0\in\res(M_\Gamma(z)+zI)\}\,.
\]
Then
\begin{lem}\label{lem:r}
$\Sigma_\Gamma\subseteq\res\cJ(\Gamma)$.
\end{lem}
\begin{proof}
\textit{Step 1.}
We show that $\sigma^0_p(\cJ(\Gamma))$ does not
intersect $\Sigma_\Gamma$.

By the definition of $M_\Gamma$ one has that
\[
\ker(M_\Gamma(z)+zI)=
\{l\vrt(l,-zl)\in\dom(\Gamma^{-1}_z\vrt_{(-z)I}) \}\,,
\quad
\Gamma^{-1}_z:=(\Gamma_z)^{-1}
\]
for $z\in\bbC_*$.
Put
\[
\Sigma^\prime_\Gamma:=
\{z\in\cO_\Gamma\vrt
0\notin\sigma_p(M_\Gamma(z)+zI) \}
\]
then by the above
\[
\Sigma^\prime_\Gamma=
\{z\in\cO_\Gamma\vrt
\Gamma^{-1}_z\vrt_{(-z)I}=\{0\}\times\ker\Gamma_z \}
\]
and then by using $\ker\Gamma_z=\whfN_z(T)$ one
concludes that
\[
\Sigma^\prime_\Gamma=
\{z\in\cO_\Gamma\vrt
\Gamma^{-1}_z\vrt_{(-z)I}=\{0\} \}\,.
\]

On the other hand
\[
\fN_z(\cJ(\Gamma))=
\{(f_z,l)\vrt ((l,-zl),\whf_z)\in
\Gamma^{-1}_z\vrt_{(-z)I} \}\,,
\quad z\in\bbC
\]
hence
\[
\Sigma^\prime_\Gamma=
\cO_\Gamma\setm\sigma^0_p(\cJ(\Gamma))\,.
\]
But
$\Sigma_\Gamma\subseteq
\Sigma^\prime_\Gamma$, so
\[
\Sigma_\Gamma\subseteq
\cO_\Gamma\setm\sigma^0_p(\cJ(\Gamma))
\quad\text{and then}\quad
\Sigma_\Gamma\subseteq
\Sigma_\Gamma\setm\sigma^0_p(\cJ(\Gamma))
\]
\ie
$\Sigma_\Gamma\mcap\sigma^0_p(\cJ(\Gamma))=\emptyset$.

\textit{Step 2.}
We prove that $W_z:=\ran(\cJ(\Gamma)-zI)$ is full
for $z\in\Sigma_\Gamma$.

By noting that
\[
W_z\mcap(\{0\}\times\fL)=
\{0\}\times\ran(M_\Gamma(z)+zI)
\]
for $z\in\bbC_*$, it follows that
\[
W_z\supseteq\{0\}\times\fL
\]
for $z\in\Sigma_\Gamma$;
\ie $W_z$ for $z\in\Sigma_\Gamma$
is a singular relation from $\fH$ to $\fL$.
Then (\eg \cite{Hassi07})
$\ol{W_z}=\ol{\dom}W_z\times\fL$.
But $\dom W_z=\ran(A_*-zI)=\fH$, so
$W_z=\ol{W_z} =\fH\times\fL$.
\end{proof}
Before we state and prove the main result in the present
section we recall that, if $V$ is a standard unitary
operator in a $\whJ^\circ_\fL$-space then
$V\Gamma$ is a \br for $T^+$, and conversely;
see Section~\ref{sec:VG}
for more details. Therefore, letting ($I=I_\fL$)
\begin{equation}
V_\varepsilon:=\begin{pmatrix}
\varepsilon^{-1/2}I & 0 \\
0 & \varepsilon^{1/2}I
\end{pmatrix}\,,\quad
\varepsilon>0
\label{eq:Ve}
\end{equation}
and $\Gamma_\varepsilon:=V_\varepsilon\Gamma$
we conclude that $\Gamma_\varepsilon$ is a \br
for $T^+$. Let us define a nonempty open set
\begin{equation}
B^\varepsilon_T:=\{z\in\delta_T
\vrt\;\abs{z}>\varepsilon\}
\label{eq:Be}
\end{equation}
where $\delta_T\neq\emptyset$ is as in \eqref{eq:dT}.
Then
\begin{lem}\label{lem:r2}
$B^\varepsilon_T\subseteq\Sigma_{\Gamma_\varepsilon}$.
\end{lem}
\begin{proof}
Fix $z\in B^\varepsilon_T$.
Then $M_\Gamma(z)\in\wtscrC(\fL)$
admits a canonical decomposition
\[
M_\Gamma(z)=M^s_\Gamma(z)\hop
(\{0\}\times \fL_z)\,,\quad
\fL_z:=\mul M_\Gamma(z)=\ol{\fL_z}
\]
where $M^s_\Gamma(z)$ is a closed operator
in $\fL$ with
$\dom M^s_\Gamma(z)=\dom M_\Gamma(z)$
(\eg \cite{Hassi18,Hassi07}).

\textit{Step 1.}
We prove that $M_{\Gamma_\varepsilon}(z)+zI$
is surjective for $z\in B^\varepsilon_T$.

Let $U_z$ be a unitary operator from
a Hilbert space $\fL$ to a Hilbert space
$(\dom M_\Gamma(z),\braket{\cdot,\cdot}_z)$,
\[
\braket{l,l^\prime}_z:=
\braket{l,l^\prime}_\fL+
\braket{M^s_\Gamma(z)l,M^s_\Gamma(z)l^\prime}_\fL\,,
\quad l,l^\prime\in\dom M_\Gamma(z)\,.
\]
Put
$\wtM^s_\Gamma(z):=M^s_\Gamma(z)U_z$.
Then $(\forall l\in\fL)$
\[
\norm{l}^2_\fL=
\norm{U_zl}^2_z=
\norm{U_zl}^2_\fL+
\norm{\wtM^s_\Gamma(z)l}^2_\fL
\]
so $\wtM^s_\Gamma(z)\in[\fL]$ with
the operator norm
$\norm{\wtM^s_\Gamma(z)}\leq1$.

Let $M_{\Gamma_\varepsilon}$ be the Weyl family
corresponding to a \ubp $(\fL,\Gamma_\varepsilon)$
for $T^+$. Then
\[
M_{\Gamma_\varepsilon}(z)=M^s_{\Gamma_\varepsilon}(z)
\hop(\{0\}\times \varepsilon^{1/2}\fL_z)
\]
with the operator part
$M^s_{\Gamma_\varepsilon}(z)=
\varepsilon M^s_{\Gamma}(z)$. Put
\[
\wtM^s_{\Gamma_\varepsilon}(z):=
M^s_{\Gamma_\varepsilon}(z)U_z=
\varepsilon\wtM^s_\Gamma(z)
\]
so that
$\wtM^s_{\Gamma_\varepsilon}(z)\in[\fL]$
and
$\norm{\wtM^s_{\Gamma_\varepsilon}(z)}\leq
\varepsilon $.

By the Neumann series for bounded operators
\[
0\in\res(\wtM^s_{\Gamma_\varepsilon}(z)+zU_z)\,,
\quad z\in B^\varepsilon_T\,.
\]
Now that
\[
\ran(M_{\Gamma_\varepsilon}(z)+zI)=
\ran(M^s_{\Gamma_\varepsilon}(z)+zI)+\fL_z\,,
\]
\[
\ran(M^s_{\Gamma_\varepsilon}(z)+zI)=
\ran(\wtM^s_{\Gamma_\varepsilon}(z)+zU_z)
\]
it follows that
$M_{\Gamma_\varepsilon}(z)+zI$ is surjective for
$z\in B^\varepsilon_T$.

\textit{Step 2.}
We prove that $M_{\Gamma_\varepsilon}(z)+zI$
is injective, and then bijective,
for $z\in B^\varepsilon_T$.

Let $z\in B^\varepsilon_T$ and
let $\Pi_z$ be an orthogonal projection in
$\fL$ onto $\fL_z$. Then
\begin{align*}
\ker(M_{\Gamma_\varepsilon}(z)+zI)=&
\ker(M^s_{\Gamma_\varepsilon}(z)+zI-z\Pi_z)
\\
=&U_z\ker(\wtM^s_{\Gamma_\varepsilon}(z)+zU_z-z\Pi_zU_z)\,.
\end{align*}
Using
\[
\wtM^s_{\Gamma_\varepsilon}(z)+zU_z-z\Pi_zU_z=
zU_z(S_z+I)-z\Pi_zU_z\,,
\]
\[
S_z:=z^{-1}U^*_z\wtM^s_{\Gamma_\varepsilon}(z)\,,
\quad
\norm{S_z}<1\quad\text{so}\quad
-1\in\res S_z
\]
it follows that
\[
\wtM^s_{\Gamma_\varepsilon}(z)+zU_z-z\Pi_zU_z=
zU_z(S_z+I)
[I-(S_z+I)^{-1}U^*_z\Pi_zU_z]\,.
\]
Thus by the Neumann series
$M_{\Gamma_\varepsilon}(z)+zI$ is injective
for $z\in B^\varepsilon_T$ such that
\[
\norm{(S_z+I)^{-1}U^*_z\Pi_zU_z}<1\,.
\]
But
\[
\norm{(S_z+I)^{-1}}<1\,,\quad
\norm{U^*_z\Pi_zU_z}\leq1
\]
so the above inequality holds true indeed.
\end{proof}
Subsequently Lemmas~\ref{lem:r} and \ref{lem:r2}
(see also Theorem~\ref{thm:IUBP2xxcor})
imply the next result.
\begin{thm}\label{thm:resTG}
With a standard unitary operator
$V_\varepsilon$ in \eqref{eq:Ve},
a \br $\Gamma$ for $T^+$
is in 1-1 correspondence with a \br
$\Gamma_\varepsilon:=V_\varepsilon\Gamma$
for $T^+$ such that $\res\cJ(\Gamma_\varepsilon)
\supseteq B^\varepsilon_T$, where
$B^\varepsilon_T$ is defined in \eqref{eq:Be}.
\end{thm}
Therefore, a \br $\Gamma$ for $T^+$ such that
$\delta_T\neq\emptyset$ is in 1-1 correspondence
with a \br $\Gamma^\prime$ for $T^+$ such that
$\res\cJ(\Gamma^\prime)\neq\emptyset$.

As an application we show that $M_\Gamma$
is a generalized Nevanlinna family provided
$T$ is standard in a Pontryagin space;
recall Example~\ref{exam:Pstan}.
We recall, see \eg \cite[Definition~4.2]{Behrndt11},
that a family, $M$, of linear relations
$M(z)$ in $\fL$ defined for $z\in\cO$, where
$\cO$ is an open symmetric subset of $\bbC_*$,
is called a \textit{generalized Nevanlinna family}
with $\kappa$ negative squares if the following
conditions hold:
\begin{enumerate}[label=(\arabic*),
ref=(\arabic*)]
\item\label{item:M1}
$M(z)^*=M(\ol{z})$ for all $z\in\cO$.
\item\label{item:M2}
$(M(z)+wI)^{-1}\in[\fL]$ for all $z\in\bbC_+\mcap\cO$
and some $w\in\bbC_+$.
\item\label{item:M3}
The kernel
$(\Psi(\ol{z})^*\Phi(\ol{w})-\Phi(\ol{z})^*\Psi(\ol{w}))
/(z-\ol{w})$ has $\kappa$ negative squares
(see \eg \cite{Alpay98} for details) for all
$z$, $w\in\cO$, where
\[
\Phi(z):=\begin{cases}
-(M(z)+wI)^{-1}\,, & z\in\bbC_+\mcap\cO\,,
\\
-(M(z)+\ol{w}I)^{-1}\,, & z\in(\bbC_+\mcap\cO)^*
\end{cases}
\]
and $\Psi(z):=-(I+z\Phi(z))$.
\end{enumerate}
\begin{exam}\label{exam:Pstan2}
Let $\fH$ be a Pontryagin space
and let $T\in\wtscrC_s(\fH)$ be standard.
The Weyl family $M_\Gamma$ corresponding to
a \br $\Gamma$ for $T^+$ is a generalized
Nevanlinna family.
\end{exam}
\begin{proof}
By Corollary~\ref{cor:rrzz}, condition \ref{item:M1}
is satisfied
for $z\in\delta_T$. By Theorem~\ref{thm:resTG}
there is a $\varepsilon>0$
such that the Weyl family
$M_{\Gamma_\varepsilon}$ corresponding to a
\br $\Gamma_\varepsilon$ for $T^+$
satisfies condition \ref{item:M2}
with $z=w\in B^\varepsilon_T$.
But
\[
(M_\Gamma(z)+\varepsilon^{-1}zI)^{-1}
=(\varepsilon^{1/2}I)(M_{\Gamma_\varepsilon}(z)+zI)^{-1}
(\varepsilon^{1/2}I)
\]
so $M_\Gamma$ on $B^\varepsilon_T$
fulfills \ref{item:M2} too.

By using
\[
\Phi_{\Gamma_\varepsilon}(z):=
P_\fL(\cJ(\Gamma_\varepsilon)-zI)^{-1}E_\fL=
-(M_{\Gamma_\varepsilon}(z)+zI)^{-1}\,,
\]
\[
P_\fL\co\fH\times\fL\lto\fL\,,\quad
(f,l)\mapsto l\,;\quad
E_\fL\co\fL\lto\{0\}\times\fL\,,\quad
l\mapsto(0,l)
\]
and letting $\Phi=\Phi_{\Gamma_\varepsilon}$
it follows that,
for $n\in\bbN$, $\lambda_\alpha\in
B^\varepsilon_T$,
$l_\alpha\in\fL$, the Hermitian matrix
(see also \cite[Theorem~4.8]{Behrndt11})
\[
([P_\fH(\cJ(\Gamma_\varepsilon)-\ol{\lambda_\beta})^{-1}
(0,l_\alpha),
P_\fH(\cJ(\Gamma_\varepsilon)-\ol{\lambda_\alpha})^{-1}
(0,l_\beta)] )^n_{\alpha,\beta=1}
\]
where $P_\fH\co\fH\times\fL\lto\fH$,
$(f,l)\mapsto f$, and
\[
P_\fH(\cJ(\Gamma_\varepsilon)-\ol{\lambda_\alpha})^{-1}
(0,l_\beta)\in\fN_{\ol{\lambda_\alpha}}(A_*)
\]
has $\kappa^\prime\leq\kappa$ negative squares;
$\kappa$ is the negative index of $\fH$.
This shows that $M_{\Gamma_\varepsilon}$, and
then $M_\Gamma$,
on $B^\varepsilon_T$ satisfies condition \ref{item:M3},
\ie $M_{\Gamma}$
is a generalized Nevanlinna family with $\kappa^\prime$
negative squares.
\end{proof}
In the above example letting $T$ be simple
we deduce \cite[Theorem~4.8]{Behrndt11}.
\section{Transformation of a boundary pair
by means of
\texorpdfstring{$\Gamma\lto\Gamma V^{-1}$}{}}\label{sec:GV}
In the present and subsequent sections
we examine Weyl families corresponding to
the transformed isometric/(essentially) unitary boundary pairs
(for short, \ibp\!\!/\eeubp\!\!).
We are mainly interested in the cases for which
Lemmas~\ref{lem:Derk}, \ref{lem:Torth} are applicable.

Let $\fL$ be a Hilbert space and let
$\fH$ and $\cH$ be Krein spaces with fundamental
symmetries $J_\fH=J$ and $J_\cH$, respectively.
\begin{thm}\label{thm:IUBP}
Let $T\in\wtscrC_s(\fH)$, let
$(\fL,\Gamma)$ be an \ibp\!\!/\eeubp for $T^+$,
and let $V$ be a unitary relation
$\whJ_\fH\lto\whJ_\cH$
with a closed $\dom V\supseteq A_*$.
Put $T^\prime:=V(T)$ and $\Gamma^\prime:=\Gamma V^{-1}$.
Then: $T^\prime\in\wtscrC_s(\cH)$,
$(\fL,\Gamma^\prime)$ is an \ibp\!\!/\eeubp for the
adjoint $T^{\prime\,+}=V(T^+)$, and the
corresponding Weyl family $M_{\Gamma^\prime}$
is given by
\[
M_{\Gamma^\prime}(z)=\Gamma(\whfN^V_z(A_*))\,,\quad
z\in\bbC_*
\]
where
\[
\whfN^V_z(A_*):=\dom( V\mcap(A_*\times zI) )\,.
\]

If moreover
$\ker V\subseteq\ker\Gamma$
then $V$ establishes a 1-1
correspondence between an \ibp\!\!/\eeubp $(\fL,\Gamma)$
for $T^+$ and an \ibp\!\!/\eeubp $(\fL,\Gamma^\prime)$
for $T^{\prime\,+}$.
\end{thm}
\begin{proof}
$T^\prime\in\wtscrC_s(\cH)$ and
$T^{\prime\,+}=V(T^+)$ follow from Lemma~\ref{lem:Torth}.
By Lemma~\ref{lem:Derk}
the adjoint
$\Gamma^{\prime\,+}=V\Gamma^+$, so
$\mul\Gamma^{\prime\,+}=T^\prime$,
and this proves $(\fL,\Gamma^\prime)$ is an
\ibp for $T^{\prime\,+}$.

Suppose $(\fL,\Gamma)$ is an \eubp for $T^+$;
hence $\ol{\Gamma}=\Gamma_\#$,
$\Gamma^{\prime\,+}=(\ol{\Gamma}V^{-1})^{-1}$.
Since also $A_\star\subseteq\dom V$, it holds
$(\ol{\Gamma}V^{-1})^+=(\ol{\Gamma}V^{-1})^{-1}$
by Lemma~\ref{lem:Derk}, and this proves
$(\fL,\Gamma^\prime)$ is an \eubp for
$T^{\prime\,+}$.

Put $(I=I_\cH)$
\[
A^\prime_*:=\dom\Gamma^\prime\,,\quad
V_z:=V\mcap(A_*\times zI)\,.
\]
Then
\[
A^\prime_*=V(A_*)\,,\quad
\whfN_z(A^\prime_*)=\ran V_z=V_z(A_*)\,;
\]
note that $\dom V_z\subseteq A_*$. Then the relation
\[
M_{\Gamma^\prime}(z):=\Gamma^\prime(\whfN_z(A^\prime_*))
=\Gamma(A_*\mcap V^{-1}V_z(A_*))\,.
\]

$V^{-1}V_z\in\scrC(\fH^2)$ consists of
$(\whf,\whg)\in\fH^4$ such that $\whf\in A_*$
and $(\exists h\in\cH)$
$(\whf,(h,zh))\in V$, $(\whg,(h,zh))\in V$.
It follows that
$\whg-\whf\in\ker V$, so
\[
V^{-1}V_z=I_{\dom V_z}\hsum(\{0\}\times\ker V)
\]
and then
\[
V^{-1}V_z(A_*)=\dom V_z\hsum\ker V\,.
\]
By noting that
\[
\ker V_z=A_*\mcap\ker V\subseteq\dom V_z
\]
one concludes that $M_{\Gamma^\prime}(z)=\Gamma(\dom V_z)$.

Finally, we verify that
an \ibp\!\!/\eeubp $(\fL,\Gamma^\prime)$
for $T^{\prime\,+}$ leads back to an \ibp\!\!/\eeubp
$(\fL,\Gamma)$ for $T^+$. By hypothesis $T$
satisfies \eqref{eq:kerVTdomV}, so
$T=V^{-1}(T^\prime)$.
Since $\ran V\supseteq T^{\prime\,+}\supseteq T^\prime$
is closed, Lemma~\ref{lem:Torth}
gives $T\in\wtscrC_s(\fH)$,
$T^+=V^{-1}(T^{\prime\,+})$. Moreover
\[
\Gamma^\prime V=\Gamma V^{-1}V=\Gamma\hsum
(\ker V\times\{0\})=\Gamma
\]
and then by Lemma~\ref{lem:Derk}
$\Gamma^+=V^{-1}\Gamma^{\prime\,+}$ and
$\ol{\Gamma}=\ol{\Gamma^{\prime}}V$,
since
$\dom\Gamma^\prime=V(A_*)\subseteq\ran V$
and
$\ran\Gamma^{\prime\,+}=V(A_\#)\subseteq\ran V$,
respectively.
It follows that $\mul\Gamma^+=V^{-1}(T^\prime)=T$,
so $(\fL,\Gamma)$ is an \ibp\!\!/\eeubp for $T^+$.
\end{proof}
\begin{rem}
Suppose $(\fL,\Gamma)$ is a \ubp for $T^+$.
The assumption $A_*\subseteq\dom V\in\wtscrC(\fH)$
to make $\Gamma^\prime$ unitary in
Theorem~\ref{thm:IUBP} can be weakened
by requiring instead that both $\dom V$ and
$A_*\hsum\dom V$ be in $\wtscrC(\fH)$. This is shown
by applying Theorem~2.13(iv) in the arXiv version
of \cite{Derkach09} and then Lemma~\ref{lem:VWorth}.
But in this case
$T^{\prime\,+}\supseteq V(T^+)$, with the equality
if \ref{item:V} holds.
\end{rem}
It follows that a 1-1
correspondence between a \br $\Gamma$ for $T^+$
and a \br $\Gamma^\prime$ for $T^{\prime\,+}$ is established
via a unitary relation $V$ with closed domain
and satisfying $\ker V\subseteq T$. The special
case of the present statement is considered in
\cite{Behrndt11}.
\begin{exam}\label{exam:fT}
With a standard unitary operator
$V=U_{J}:=\bigl(\begin{smallmatrix}I & 0 \\ 0 & J
\end{smallmatrix}\bigr)\in[\fH^2]$
the correspondence principle between a \br
$\Gamma$ for $T^+$ and a \br $\Gamma^\prime$ for
$T^{\prime\,*}=J T^+$
is formulated in \cite[Proposition~3.3]{Behrndt11}.
For the present $V$ it holds
\[
T^\prime=J T=:\fT\,,\quad
\whfN^{V}_z(A_*)=J\whfN_z(J A_*)\,.
\]

Assume that $T$ is a densely defined,
closed, symmetric operator in a $J$-space, and that
a \ubp $(\fL,\Gamma)$ turns into an
\obt $\Pi_\Gamma=(\fL,\Gamma_0,\Gamma_1)$
for $T^+$. Let $\gamma_\Gamma$ and $M_\Gamma$
be the $\gamma$-field and Weyl family $M_\Gamma$
corresponding to $\Pi_\Gamma$. By identifying
$T^+$ with $\dom T^+=\dom\fT^*$,
$\Pi_{\Gamma^\prime}=\Pi_\Gamma$ is an \obt for
$\fT^*$ with Weyl family
\[
M_{\Gamma^\prime}(z)=\{(\Gamma_0f_z,\Gamma_1f_z)\vrt
f_z\in\fN_z(\fT^*)\}\,,\quad z\in\bbC_*\,.
\]
It will follow from Theorem~\ref{thm:IUBP3} that,
if $\res T_0\mcap\res\fT_0\neq\emptyset$, $\fT_0:=J T_0$,
then for $z$ from this set, the
Weyl function can be represented
on all of $\fL$ by using \eqref{eq:fTex}.

Let $P_\pm$ be the canonical
projections in $\fH$ onto $\fH_\pm$ (positive and
negative subspaces of a $J$-space). Since
$I-J=2P_-$, it will follow from
Corollary~\ref{cor:Delta0} that it may happen
that $M_{\Gamma^\prime}(z)=M_{\Gamma}(z)$
iff either
(a) $z=0\in\res T_0\mcap\res\fT_0$ or
(b) $z\in\res T_0\mcap\res\fT_0\setm\{0\}$ and
$\fN_z(T^+)\subseteq\fH_+$.
Note that (a)--(b) are equivalent to
$\fN_z(T^+)\subseteq\fN_z(\fT^*)$.
\end{exam}
\subsection{Weyl family of transformed
\texorpdfstring{\obt}{} }\label{sec:LFT}
By applying Theorem~\ref{thm:IUBP}
we characterize the Weyl
family corresponding to the transformed \obt
More specifically, we assume that $T$
is a densely defined, closed,
symmetric operator in a Krein space $\fH$ with
a fundamental symmetry $J_\fH=J$ and that
$\Pi_\Gamma=(\fL,\Gamma_0,\Gamma_1)$ is an
\obt for $T^+$, so that
$\dom\Gamma=T^+$ is identified with $\dom T^+$.
We transform $\Gamma$ according to
$\Gamma\lto\Gamma^\prime:=\Gamma V^{-1}$,
where $V$ is a standard unitary operator from a
$\whJ_\fH$-space to a $\whJ_\cH$-space,
which we represent in matrix form.
By Theorem~\ref{thm:IUBP}, and noting that
$\Gamma^\prime$ is an operator with
$\ran\Gamma^\prime=\ran\Gamma$, it follows that
the triple
$\Pi_{\Gamma^\prime}=
(\fL,\Gamma^\prime_0,\Gamma^\prime_1)$
with the operators
\[
\Gamma^\prime_i:=
\left\{\left(V\begin{pmatrix}
f \\ T^+f\end{pmatrix},\Gamma_if \right)\vrt
f\in\dom T^+ \right\}\,,\quad i\in\{0,1\}
\]
is an \obt for the adjoint $T^{\prime\,+}=V(T^+)$ of
a closed symmetric relation
\[
T^\prime:=V(T)=\left\{V\begin{pmatrix}
f \\ Tf\end{pmatrix}\vrt f\in\dom T \right\}
\]
in a $J_\cH$-space.
Generally, $T^\prime\in\wtscrC_s(\cH)$ is neither
an operator nor densely defined, as
$\mul T^\prime$ consists of $h^\prime\in\cH$ such that
$(\exists\whf\in T)$ $\bigl(\begin{smallmatrix}
0 \\ h^\prime
\end{smallmatrix}\bigr)=V\whf$,
and similarly for $\mul T^{\prime\,+}$.
Under additional assumptions imposed on $V$
we achieve, however, that $T^\prime$ is a densely defined
operator. Having done that, we concentrate on
the characterization of the Weyl family
$M_{\Gamma^\prime}$ corresponding to $\Pi_{\Gamma^\prime}$.
Theorem~\ref{thm:IUBP3} is our main result in this
paragraph.
\subsubsection{Linear fractional transformation}
Our consideration is based on a general hypothesis
that $V(T)$ is a linear fraction transformation
of $T$. We now discuss this in more detail.

A standard unitary operator $V\in[\fH^2,\cH^2]$
from a $\whJ_\fH$-space to a $\whJ_\cH$-space
is represented in the matrix form
\begin{equation}
V=\begin{pmatrix}
A & B \\ C & D
\end{pmatrix}\,,
\quad
V^{-1}=V^+=\begin{pmatrix}
D^+ & - B^+ \\
-C^+ & A^+
\end{pmatrix}
\label{eq:stan}
\end{equation}
where the operators $A$, $B$, $C$, $D\in[\fH,\cH]$ satisfy
the conditions
\begin{equation}
\begin{split}
A^+D-C^+B=I_\fH\,,&\quad
AD^+-BC^+=I_\cH\,,
\\
A^+C\in[\fH]\;\;\text{is ($J_\fH$-)self-adjoint,}&\quad
AB^+\in[\cH]\;\;\text{is ($J_\cH$-)self-adjoint,}
\\
B^+D\in[\fH]\;\;\text{is ($J_\fH$-)self-adjoint,}&\quad
CD^+\in[\cH]\;\;\text{is ($J_\cH$-)self-adjoint.}
\end{split}
\label{eq:dkfjrnw0}
\end{equation}
The prefix $J_\fH$- (resp. $J_\cH$-) emphasizes
that, for example, $A^+C$ is self-adjoint in
a Krein space $\fH$.
The reader may compare $V$ in \eqref{eq:stan}
with \cite[Corollary~3.4]{Derkach15a},
\cite[Eqs.~(4.1)--(4.3)]{Behrndt09}.
\begin{exam}[Hilbert space unitary operator]\label{exam:fT2}
One verifies that,
if $V$ is a standard unitary operator
such that
\begin{equation}
\whJ_\cH V=V\whJ_\fH
\label{eq:commutV}
\end{equation}
then $V$ is a Hilbert space unitary operator,
$V^{-1}=V^*$.
Conversely, if $V$ is a Hilbert space
unitary operator such that \eqref{eq:commutV} holds,
then $V$ is a standard unitary operator.
A general matrix form of an operator $V$ satisfying
\eqref{eq:commutV} is
\begin{equation}
V=V(A,B;J_\fH,J_\cH):=
\begin{pmatrix}A & B \\ -J_\cH B J_\fH & J_\cH A J_\fH
\end{pmatrix}
\label{eq:VABJJ}
\end{equation}
for some $A$, $B\in[\fH,\cH]$ (\cf \eqref{eq:stan}).
If, moreover, $V$ in \eqref{eq:VABJJ} is a Hilbert
space unitary operator, that is, if $V$ is a standard unitary
operator,
then the operators $A$ and $B$ satisfy
(put $C=-J_\cH BJ_\fH$ and $D=J_\cH AJ_\fH$ in \eqref{eq:dkfjrnw0}):
\[
\begin{split}
A^*A+J_\fH B^*BJ_\fH=I_\fH\,,&\quad
AA^*+BB^*=I_\cH\,,
\\
A^*B\in[\fH]\;\;\text{is ($J_\fH$-)self-adjoint,}&\quad
AB^+\in[\cH]\;\;\text{is ($J_\cH$-)self-adjoint.}
\end{split}
\]
The inverse of $V$ is then given by
\[
V^{-1}=V^*=V^+=V(A^*,-B^+;J_\cH,J_\fH)=
\begin{pmatrix}A^* & -B^+ \\ B^* & A^+
\end{pmatrix}\,.
\]
Letting $J_\fH=I_\fH$ or $J_\cH=I_\cH$ or both,
one constructs from the above unitary operators
$\whJ^\circ_\fH\lto\whJ_\cH$ or
$\whJ_\fH\lto\whJ^\circ_\cH$ or
$\whJ^\circ_\fH\lto\whJ^\circ_\cH$,
which are simultaneously Hilbert space unitary operators.
Thus,
a Hilbert space unitary operator
(\cite[Eq.~(3.3)]{Behrndt11})
\[
U_{J}:=V(I,0;J,I)=
V(I,0;I,J)=
\begin{pmatrix}I & 0 \\ 0 & J
\end{pmatrix}
\]
in $\fH^2$ is unitary as an operator
$\whJ_\fH\lto\whJ^\circ_\fH$ or
$\whJ^\circ_\fH\lto\whJ_\fH$.
Note that
\[
V(A,B;J_\fH,J_\cH)=U_{J_\cH}V^\prime\,,\quad
V^\prime=V(A,B;J_\fH,I_\cH)
\]
with
a unitary (and Hilbert unitary)
$V^\prime\co\whJ_\fH\lto\whJ^\circ_\cH$.
The operator $V^\prime$ further decomposes as
\[
V^\prime=V^{\prime\prime}U_{J_\fH}\,,\quad
V^{\prime\prime}=V(A,BJ_\fH;I_\fH,I_\cH)
\]
with a unitary (and Hilbert unitary)
$V^{\prime\prime}\co\whJ^\circ_\fH\lto\whJ^\circ_\cH$.
\end{exam}
Let $V$ be a standard unitary operator as
in \eqref{eq:stan}
(with $A,\ldots,D$ as in \eqref{eq:dkfjrnw0}).
Let $T\in\scrC(\fH)$ be an operator.
Let
\[
W(A,B;T):=(A+BT)\vrt_{\dom T}\,.
\]
Observe that $W(A,0;T)\subseteq A$.
In the definition of $W(\cdots)$ the arguments
may vary.

Let $T^\prime:=V(T)\in\scrC(\cH)$.
Then $T^\prime$ is an operator iff
\begin{equation}
\ker W(A,B;T)\subseteq\ker W(C,D;T)\,.
\label{eq:kerker}
\end{equation}
Assume \eqref{eq:kerker}.
Because $V^{-1}(T^\prime)=T$
(recall \eqref{eq:kerVTdomV}), it follows that
\begin{equation}
\begin{split}
\dom T=\ran W(D^+,-B^+;T^\prime)\,,
\quad
\dom T^\prime=\ran W(A,B;T)\,,
\\
\ran T=\ran W(-C^+,A^+;T^\prime)\,,
\quad
\ran T^\prime=\ran W(C,D;T)\,.
\end{split}
\label{eq:tratarete}
\end{equation}
Moreover,
since $T$ is an operator, it holds
\[
\ker W(D^+,-B^+;T^\prime)\subseteq
\ker W(-C^+,A^+;T^\prime)
\]
which,
in view of \eqref{eq:dkfjrnw0}, is a consequence of
\eqref{eq:kerker}. From \eqref{eq:dkfjrnw0} and
\eqref{eq:tratarete} it follows that
\begin{align*}
W(A,B;T)W(D^+,-B^+;T^\prime)=&I_{\dom T^\prime}\,,
\\
W(D^+,-B^+;T^\prime)W(A,B;T)=&I_{\dom T}\,.
\end{align*}

From the above we conclude that
$T^\prime$ is an operator, \ie
\eqref{eq:kerker} holds, iff
$W(A,B;T)$ is invertible, \ie
$0\notin\sigma_p(W(A,B;T))$. In this
case the inverse operator $W(A,B;T)^{-1}$
is given by
\begin{equation}
\begin{split}
W(A,B;T)^{-1}=&W(D^+,-B^+;T^\prime)\,,
\\
W(D^+,-B^+;T^\prime)^{-1}=&W(A,B;T)\,.
\end{split}
\label{eq:AAAWW}
\end{equation}
Thus,
if $T$ is a densely defined symmetric
operator in a $J$-space, and if
$0\notin\sigma_p(W(A,B;T^+))$, then clearly
$W(A,B;T)$ is also an invertible operator.
Under the present hypothesis, relations in
\eqref{eq:tratarete}, \eqref{eq:AAAWW} hold
with $T$ (resp. $T^\prime$) replaced
by $T^+$ (resp. $T^{\prime\,+}$).

For $0\notin\sigma_p(W(A,B;T))$,
$T^\prime:=V(T)=\phi_V(T)$
is the \textit{linear fractional transformation}
(\lft\!\!) of $T$ induced by $V$; \ie
\[
\phi_V(T):=W(C,D;T)W(A,B;T)^{-1}\,,
\quad
0\notin\sigma_p(W(A,B;T))\,.
\]
Similarly $T$ is the \lft of $T^\prime$ induced
by $V^{-1}=V^+$; \ie $T=\phi_{V^{-1}}(T^\prime)$.

Linear fractional transformations of contractions
$T$ with $0\in\res W(A,B;T)$
are classical objects extensively studied \eg in
\cite{Shmulyan80,Shmulyan80a,Shmulyan78,Krein74}.
\begin{exam}[\lft of $z$]\label{exam:cDVz0}
Let $V=V(A,B;I_\fH,I_\cH)$, $B\neq0$,
be a Hilbert space unitary operator as defined
in \eqref{eq:VABJJ}.
Then $0\notin\sigma_p(W(A,B;z_0I_\fH))$ for
$z_0\in\{-\img,\img\}$.
Moreover $W(A,B;z_0I_\fH)$
is a Hilbert space unitary operator
and $\phi_V(z_0I_\fH)=z_0I_\cH$. Let
$\cD_V(z_0)$ be an open disc in $\bbC$ of
center $z_0\in\{-\img,\img\}$ and radius
$\norm{B}^{-1}$ ($0<\norm{B}\leq1$ is the operator
norm of $B$). For $z\in\bbC_*\mcap\cD_V(z_0)$,
the \lft $\phi_V(zI_\fH)$ of $zI_\fH$ is the Nevanlinna
function $V(zI_\fH)$ induced by the Nevanlinna pair
$(W(A,B;\cdot),W(-B,A;\cdot))$. The reader may look
at \eg
\cite[Definition~3.2]{Behrndt09},
\cite[Definition~2.2]{Behrndt08} for terminology.
\end{exam}
\begin{rem}
There is another representation of $\phi_V$,
which we mention here without proof, as
we do not use it in what follows.
Let $T$ be an operator in a Krein space $\fH$,
let $V$ be a standard unitary operator as in
\eqref{eq:stan}, and let $0\notin\sigma_p(W(A,B;T))$.
Put
\[
T^\prime:=\wtW T W_*\,,\quad
W_*:=W(A,B;T)^{-1}
\]
for some $\wtW\in[\fH,\cH]$.
Then $T^\prime=\phi_V(T)$ iff
$(\wtW-D)T\subseteq C$; take \eg
a Hilbert unitary $V=V(A,0;I_\fH,I_\cH)$
in \eqref{eq:VABJJ} for a quick example.
\end{rem}
Let $V$ be a standard unitary operator as previously
and let $T$ be a densely defined closed operator
in a Krein space $\fH$ such that $0\notin\sigma_p(W(A,B;T))$.
By Lemma~\ref{lem:Torth} $T^\prime=\phi_V(T)$ is
a closed operator in a Krein space $\cH$, which is
not necessarily densely defined. The adjoint
relation $T^{\prime\,+}=V(T^+)$ is an operator iff
\eqref{eq:kerker} holds with $T$ replaced by $T^+$,
\ie iff $0\notin\sigma_p(W(A,B;T^+))$.
Subsequently we have the next result, where
we let $p_1\co\fH^2\lto\fH$,
$(f,f^\prime)\mapsto f$.
\begin{lem}\label{lem:lftTTh}
Let $T$ be a densely defined, closed, symmetric
operator in a Krein space $\fH$. Let
$\Pi_\Gamma=(\fL,\Gamma_0,\Gamma_1)$ be an \obt
for $T^+$.
Let $V$ be a standard unitary operator
as in \eqref{eq:stan} and such that
$0\notin\sigma_p(W(A,B;T^+))$.
\begin{enumerate}[label=\arabic*$^\circ$,
ref=\arabic*$^\circ$]
\item\label{item:lftTTh-1o}
Let $T^\prime:=V(T)$.
Then $T^\prime=\phi_V(T)$ is the \lft of $T$,
which is a densely defined,
closed, symmetric operator in a Krein space $\cH$,
and whose adjoint operator is the \lft of $T^+$,
$T^{\prime\,+}=\phi_V(T^+)$.
The triple
\[
\Pi_{\Gamma^\prime}=
(\fL,\Gamma^\prime_0,\Gamma^\prime_1)\,,
\quad
\Gamma^\prime_i:=\Gamma_iW(A,B;T^+)^{-1}\,,\quad
i\in\{0,1\}
\]
is an \obt for $T^{\prime\,+}$ with the Weyl family
$M_{\Gamma^\prime}$ given by
\[
M_{\Gamma^\prime}(z)=
\{(\Gamma_0f_z,\Gamma_1f_z)\vrt
f_z\in\fN^V_z(T^+) \}\,,\quad z\in\bbC_*\,,
\]
\begin{align*}
\fN^V_z(T^+):=&
p_1(\whfN^V_z(T^+))=\ker p_V(z;T^+)\,,
\\
p_V(z;\cdot):=&
zW(A,B;\cdot)-W(C,D;\cdot)\,.
\end{align*}
\item\label{item:lftTTh-2o}
Consider an operator $T_\Theta\in\mrm{Ext}_T$,
$\Theta\in\scrC(\fL)$, parametrized by
\[
\dom T_\Theta:=\{f\in\dom T^+\vrt
(\Gamma_0f,\Gamma_1f)\in\Theta\}\,.
\]
Then an operator $T^\prime_\Theta\in\mrm{Ext}_{T^\prime}$
parametrized by
\[
\dom T^\prime_\Theta:=\{h\in\dom T^{\prime\,+}\vrt
(\Gamma^\prime_0h,\Gamma^\prime_1h)\in\Theta \}
\]
is the \lft of $T_\Theta$, \ie
$T^\prime_\Theta=\phi_V(T_\Theta)$.
\end{enumerate}
\end{lem}
\begin{proof}
\ref{item:lftTTh-1o}
To accomplish the proof it remains
to apply Theorem~\ref{thm:IUBP} with
\[
\whfN^V_z(T^+)=\{(f_z,T^+f_z))\vrt
f_z\in\ker p_V(z;T^+) \}\,.
\]

\ref{item:lftTTh-2o}
Since $0\notin\sigma_p(W(A,B;T^+))$
implies $0\notin\sigma_p(W(A,B;T_\Theta))$,
$\phi_V(T_\Theta)$ is well-defined.
By the definition of the \lft $\phi_V$ we have
\[
\dom \phi_V(T_\Theta)=
W(A,B;T_\Theta)\dom T_\Theta
=\ran W(A,B;T_\Theta)\,.
\]
On the other hand
\[
\dom T^\prime_\Theta=W(A,B;T^+)\dom T_\Theta=
W(A,B;T_\Theta)\dom T_\Theta
\]
so $\dom T^\prime_\Theta=\dom \phi_V(T_\Theta)$.
Finally, by the above
$(\forall h\in\dom T^\prime_\Theta)$
\[
W(A,B;T^+)^{-1}h=W(A,B;T_\Theta)^{-1}h
\]
and then
\begin{align*}
T^\prime_\Theta h=&
T^{\prime\,+}h=
\phi_V(T^+)h=W(C,D;T^+)W(A,B;T_\Theta)^{-1}h
\\
=&W(C,D;T_\Theta)W(A,B;T_\Theta)^{-1}h=\phi_V(T_\Theta)h
\end{align*}
as claimed.
\end{proof}
We emphasize that $p_V(z;T^+)$ explicitly reads
\[
p_V(z;T^+)=zA-C+(zB-D)T^+
\]
and is defined on $\dom p_V(z;T^+)=\dom T^+$.
\subsubsection{Weyl function of transformed
\texorpdfstring{\obt}{} }
Let $T_0:=\ker\Gamma_0$ and
$T^\prime_0:=\ker\Gamma^\prime_0$, so that
$T^\prime_0=\phi_V(T_0)$.
We characterize the Weyl family (now Weyl function)
$M_{\Gamma^\prime}$
based on hypothesis that the set
\[
\rho_V:=\res T_0\mcap\res T^\prime_0
\]
is nonempty.
The next lemma expresses $\res T^\prime_0$
in terms of $V$ and $T_0$.
\begin{lem}\label{lem:IUBP3-pre}
Assume the notation and hypotheses as in
Lemma~\ref{lem:lftTTh}.
\begin{enumerate}[label=\arabic*$^\circ$,
ref=\arabic*$^\circ$]
\item\label{item:IUBP3-pre-1o}
$z\in\res T^\prime_0$ iff $0\in\res p_V(z;T_0)$.
\item\label{item:IUBP3-pre-2o}
Suppose $\bbC_*\mcap\res T^\prime_0\neq\emptyset$.
$\Gamma_0$ maps $\fN^V_z(T^+)$
bijectively onto $\fL$ for $z\in\res T^\prime_0$,
and the Weyl function of $T^\prime$ is given by
\[
M_{\Gamma^\prime}(z)=\Gamma_1\gamma^V_\Gamma(z)
\in[\fL]\,,
\quad
\gamma^V_\Gamma(z):=(\Gamma_0\vrt_{\fN^V_z(T^+)})^{-1}
\in[\fL,\fH]
\]
for $z\in\bbC_*\mcap\res T^\prime_0$.
\end{enumerate}
\end{lem}
\begin{proof}
\ref{item:IUBP3-pre-1o}
This follows from
\[
T^\prime_0-zI_\cH=\phi_V(T_0)-zI_\cH
=-p_V(z;T_0)W(A,B;T_0)^{-1}\,,\quad z\in\bbC\,.
\]
\ref{item:IUBP3-pre-2o}
$\fN^V_z(T^+)$ consists of
$f\in\dom T^+$ such that
$p_V(z;T^+)f=0$. Thus
\[
\fN^V_z(T^+)=W(A,B;T^+)^{-1}\fN_z(T^{\prime\,+})\,.
\]
Since $\Gamma^\prime_0:=\Gamma_0W(A,B;T^+)^{-1}$
maps $\fN_z(T^{\prime\,+})$ bijectively onto $\fL$
for $z\in\res T^\prime_0$,
it follows that $\Gamma_0$ maps $\fN^V_z(T^+)$
bijectively onto $\fL$ for $z\in\res T^\prime_0$.
\end{proof}
\begin{thm}\label{thm:IUBP3}
Assume the notation and hypotheses as in
Lemma~\ref{lem:lftTTh}, and let $\gamma_\Gamma$
and $M_\Gamma$ be the $\gamma$-field and the
Weyl family corresponding to an \obt $\Pi_\Gamma$
for $T^+$.
Suppose $\bbC_*\mcap\rho_V\neq\emptyset$.
Then the Weyl function $M_{\Gamma^\prime}$
corresponding to an \obt
$\Pi_{\Gamma^\prime}$ for $T^{\prime\,+}$ is given
on $\fL$ by
\[
M_{\Gamma^\prime}(z)=
M_\Gamma(z)+\Delta^V_\Gamma(z)\,,
\quad z\in\rho_V\,,
\]
\begin{align*}
\Delta^V_\Gamma(z):=&
-\Gamma_1 p_V(z;T_0)^{-1}p_V(z)\gamma_\Gamma(z)\,,
\\
p_V(z):=&
p_V(z;zI_\fH)=z^2B+z(A-D)-C
\,.
\end{align*}
\end{thm}
\begin{proof}
By Lemma~\ref{lem:Behrndt20}
$f\in\fN^V_z(T^+)\subseteq\dom T^+$ is of the form
$f=f_0+f_z$ for some $f_0\in\dom T_0$,
$f_z\in\fN_z(T^+)$, and $z\in\res T_0$, so
by Lemma~\ref{lem:IUBP3-pre}-\ref{item:IUBP3-pre-2o}
$(\forall l\in\fL)$
\[
\gamma^V_\Gamma(z)l=f_0+f_z\,,\quad
f_z=\gamma_\Gamma(z)l\,,\quad
z\in\rho_V\,.
\]
On the other hand
\[
0=p_V(z;T^+)(f_0+f_z)=p_V(z;T_0)f_0+p_V(z)f_z
\]
so by Lemma~\ref{lem:IUBP3-pre}-\ref{item:IUBP3-pre-1o}
\[
f_0=-p_V(z;T_0)^{-1}p_V(z)f_z
\]
and then
\[
\gamma^V_\Gamma(z)=
[I-p_V(z;T_0)^{-1}p_V(z)]\gamma_\Gamma(z)\,.
\]
This proves $M_{\Gamma^\prime}(z)$
as stated for $z\in\bbC_*\mcap\rho_V$.
Since $\gamma^V_\Gamma$ is analytic
on $\rho_V$, $M_{\Gamma^\prime}(z)$
extends to $z\in\rho_V$.
\end{proof}
For $J_\fH=I_\fH$ and $J_\cH=I_\cH$, a classical
result due to \cite{Langer77} states that
the Weyl functions of unitarily
similar simple symmetric operators coincide.
The next proposition gives a necessary and sufficient
condition to have $M_{\Gamma^\prime}(z)=M_\Gamma(z)$
for $z\in\rho_V$ in a Krein space setting.
\begin{cor}\label{cor:Delta0}
$\Delta^V_\Gamma(z)=0$
iff $\fN_z(T^+)\subseteq\ker p_V(z)$.
\end{cor}
\begin{proof}
Using $\ran\gamma_\Gamma(z)=\fN_z(T^+)$ and
$\dom T=\ker\Gamma_0\mcap\ker\Gamma_1$ it holds
\begin{align*}
\Delta^V_\Gamma(z)=0\quad\Leftrightarrow\quad&
p_V(z;T_0)^{-1}p_V(z)\fN_z(T^+)\subseteq\dom T
\\
\Leftrightarrow\quad&
(T^\prime_0-z)^{-1}p_V(z)\fN_z(T^+)\subseteq\dom T^\prime
\\
\Leftrightarrow\quad&
p_V(z)\fN_z(T^+)\subseteq\ran(T^\prime-zI_\cH)\,.
\end{align*}
Since
\[
\ran(T^\prime-zI_\cH)=\ran p_V(z;T)
\]
we conclude that
$\Delta^V_\Gamma(z)=0$ iff
\[
p_V(z)\fN_z(T^+)=p_V(z;T^+)\fN_z(T^+)
\subseteq p_V(z;T^+)\dom T\,.
\]
Using that $\fN_z(T^+)$ and $\dom T$,
as well as $\ker p_V(z;T^+)$ and $\dom T$,
are disjoint sets,
we show that the above inclusion holds iff
$p_V(z)\fN_z(T^+)=\{0\}$.

Let $L$ be an operator in $\fH$, and let
$\fM$, $\fN\subseteq\dom L$ be disjoint
subsets, and such that $L\fM\subseteq L\fN$.
Since $(\forall m\in\fM)$ $(\exists n\in\fN)$
$m-n\in\ker L$, it follows that
$\fM\subseteq\fN+\ker L$. Thus, if the latter
sum is direct, \ie if $\fN\mcap\ker L=\{0\}$,
then $\fM\mcap\fN=\{0\}$ implies that
$\fM\subseteq\ker L$.

By the above, $\Delta^V_\Gamma(z)=0$ implies
$\fN_z(T^+)\subseteq\ker p_V(z)$. The converse
implication is clear.
\end{proof}
Under additional hypotheses there is another
variant of Corollary~\ref{cor:Delta0}. Below
we use $\norm{\cdot}$ for the operator norm.
\begin{cor}\label{cor:Delta0b}
Suppose $(\exists z_0\in\bbC)$
$0\in\res W(A,B;z_0I_\fH)$, and let $\cD_V(z_0)$ denote an
open disc in $\bbC$ of center $z_0$ and radius
$\norm{BW(A,B;z_0I_\fH)^{-1}}^{-1}$.
If $z\in\cD_V(z_0)\mcap\rho_V\neq\emptyset$ then
$\Delta^V_\Gamma(z)=0$ iff
$W(A,B;zI_\fH)\fN_z(T^+)\subseteq\fN_z(\phi_V(zI_\fH))$.
\end{cor}
\begin{proof}
This is a reformulation of Corollary~\ref{cor:Delta0},
which uses the Neumann series applied to
\begin{align*}
W(A,B;zI_\fH)=&
W(A,B;z_0I_\fH)+(z-z_0)B
\\
=&
[I_\cH+(z-z_0)BW(A,B;z_0I_\fH)^{-1}]W(A,B;z_0I_\fH)
\end{align*}
as well as the equality
\[
p_V(z)=(zI_\cH-\phi_V(zI_\fH))W(A,B;zI_\fH)\,.\qedhere
\]
\end{proof}
\begin{rem}
We emphasize that $V$ in Theorem~\ref{thm:IUBP3}
and Corollaries~\ref{cor:Delta0}, \ref{cor:Delta0b}
satisfies not only \eqref{eq:dkfjrnw0}, but also
$0\notin\sigma_p(W(A,B;T^+))$. Particularly this means
that $\fN_z(T^+)$ is nontrivial iff so is
$W(A,B;zI_\fH)\fN_z(T^+)$.
\end{rem}
In a Hilbert space Example~\ref{exam:cDVz0},
$\Delta^V_\Gamma(z)=0$
for $z\in\bbC_*$ ($\subseteq\rho_V$)
iff either $z=z_0$
or $z\in\bbC_*\setm\{z_0\}$
and $\fN_z(T^*)\subseteq\ker B$; for an invertible
operator $B$, therefore, $\Delta^V_\Gamma$ does
not vanish on $\bbC_*\setm\{z_0\}$ for $T\subsetneq T^*$.
On the other hand, if $B=0$, then
the Weyl functions of unitarily
equivalent symmetric operators $T$ and $T^\prime=ATA^*$ coincide
on $\res T_0$; note that in this case
$(A\vrt_{\dom T})^{-1}=A^*\vrt_{\dom T^\prime}$
and $\dom T^\prime=A\dom T$.
A somewhat more involved example
is motivated by \cite{Hassi13} and is
considered below.
\subsection{Example: The model for a scaled
\texorpdfstring{\obt}{}}\label{exam:iota}
Suppose $V$ in Theorem~\ref{thm:IUBP3}
is such that $A=X^{-1}$, $B=C=0$, $D=X^+$ for a bijective
$X\in[\cH,\fH]$ (and hence bijective $X^+\in[\fH,\cH]$).
The \lft of $T$ is given by
\[
T^\prime=X^+TX=X^{-1}\iota TX\,,\quad
\iota:=XX^+
\]
and is therefore (generally non-unitarily) similar to $\iota T$.
The Weyl function $M_{\Gamma^\prime}$ corresponding to an \obt
$(\fL,\Gamma^\prime_0,\Gamma^\prime_1)$
for the \lft $T^{\prime\,+}=X^+T^+X$ of $T^+$, where
\[
\Gamma^\prime_i:=\Gamma_iX\,,\quad
\dom \Gamma^\prime_i=\dom T^{\prime\,+}=
X^{-1}\dom T^+\,,\quad
i\in\{0,1\}
\]
is given by
\begin{equation}
\begin{split}
M_{\Gamma^\prime}(z)=&M_\Gamma(z)+\Delta^V_\Gamma(z)\,,\quad
z\in\rho_V=\res T_0\mcap\res(\iota T_0)\,,
\\
\Delta^V_\Gamma(z)=&z\Gamma_1(\iota T_0-z)^{-1}(I-\iota)
\gamma_\Gamma(z)
\end{split}
\label{eq:solvef00}
\end{equation}
provided $\bbC_*\mcap\rho_V\neq\emptyset$. In this
case $\Delta^V_\Gamma(z)=0$ iff
\[
z(I-\iota)\fN_z(T^+)=\{0\}\,.
\]
Note that $\iota=I$ iff $X$ is a unitary
operator $J_\cH\lto J_\fH$.

At some point it will be more convenient to
express the operator $\Delta^V_\Gamma(z)$
in terms of $(T_0-z)^{-1}$ rather than
$(\iota T_0-z)^{-1}$, namely $(\forall l\in\fL)$
\begin{equation}
\begin{split}
\Delta^V_\Gamma(z)l=\Gamma_1f_0
\text{ where $f_0\in\dom T_0$ solves}
\\
f_0=(T_0-z)^{-1}(I-\iota)(zf_z+T_0f_0)
\end{split}
\label{eq:solvef0}
\end{equation}
with $f_z:=\gamma_\Gamma(z)l$. We remark that
the solution $f_0$ is unique.
\subsubsection{Constructing \texorpdfstring{$X$}{}}
Suppose a subspace $\fH_s$ of $\fH$ is projectively complete
\cite[Definition~1.7.8,
Theorem~1.7.16]{Azizov89}, \ie
$\fH=\fH_s[\dsum]\fH_a$ is the direct $J$-orthogonal
sum of the Krein spaces $\fH_s$ and $\fH_a:=\fH^{[\bot]}_s$.
By the projective completeness,
there is a $J$-orthogonal
projection $P_s$ uniquely projecting
an element from $\fH$ to an element from $\fH_s$,
\cite[Lemma~1.7.12]{Azizov89}.
We put $P_a:=I-P_s$, and in what follows
we let $\cH=\fH$, $J_\cH=J_\fH=J$.

We consider a bijective operator
$X=X^+\in[\fH]$ defined by
\[
X:=\varkappa P_s+\varkappa^{-1}P_a\,,\quad
\varkappa\in\bbR\setm\{-1,0,1\}
\]
and having the inverse
$X^{-1}=\varkappa^{-1} P_s+\varkappa P_a$.
We exclude the values $\varkappa=-1$ and
$\varkappa=1$ as they lead to trivial results.
\subsubsection{Conditions for the existence of
a fixed point of the \texorpdfstring{\lft}{} }
We now specify the conditions imposed
on $\fH_s$, $\fH_a$ for having $T^\prime=T$ with
$X$ as defined above.
\begin{prop}\label{prop:TTp}
Assume that
\[
\begin{split}
P_s\dom T^+\subseteq\dom T_0\,,\quad
P_a\dom T^+\subseteq\dom T_1\,,
\\
T_0P_s\dom T_0\subseteq\fH_a\,,\quad
T_1P_a\dom T_1\subseteq\fH_s\,.
\end{split}
\]
Then $\Gamma_0=\Gamma_0P_a$,
$\Gamma_1=\Gamma_1P_s$,
\[
\Gamma^\prime_0=\varkappa^{-1}\Gamma_0\,,
\quad
\Gamma^\prime_1=\varkappa\Gamma_1
\quad\text{on}\quad
\dom T^{\prime\,+}=\dom T^+
\]
and an operator $T^\prime_\Theta\in\mrm{Ext}_T$,
$\Theta\in\scrC(\fL)$, in
Lemma~\ref{lem:lftTTh}-\ref{item:lftTTh-2o}
is represented by
\[
T^\prime_\Theta=T_{\Theta^\prime}\in\mrm{Ext}_T
\quad\text{with}\quad
\Theta^\prime:=(\varkappa^{-1}I)\Theta
(\varkappa^{-1}I)
\]
(with $I=I_\fL$).
Particularly $\Theta^\prime=\varkappa^{-2}\Theta$
if $\Theta$ is an operator.
\end{prop}
\begin{proof}
By hypotheses
$P_{s,a}\dom T^+\subseteq\dom T^+$ it holds
(\cf \eqref{eq:PTequiv}--\eqref{eq:domVT})
\[
\begin{split}
\dom T^+=&P_s\dom T^+[\dsum]P_a\dom T^+\,,
\\
P_s\dom T^+=&\fH_s\mcap\dom T^+\,,\quad
P_a\dom T^+=\fH_a\mcap\dom T^+\,.
\end{split}
\]
From here it follows that
\[
\dom T^{\prime\,+}=X^{-1}\dom T^+=\dom T^+\quad
\text{and}\quad
\dom T^\prime_\Theta=\dom T_{\Theta^\prime}\,.
\]

It remains to verify that
$T^{\prime\,+}f=T^+f$ for all
$f\in\dom T^+$. Using
$f=f_s+f_a$ with
$f_{s,a}\in P_{s,a}\dom T^+$ it holds
\[
T^{\prime\,+}f=X^+T^+Xf=
XT^+(\varkappa f_s+\varkappa^{-1}f_a)\,.
\]
Since $f_s\in\dom T_0$,
$T^+f_s=T_0f_s\in\fH_a$ by hypothesis
$T_0P_s\dom T_0\subseteq\fH_a$.
Similarly $T^+f_a=T_1f_a\in\fH_s$ by
$T_1P_a\dom T_1\subseteq\fH_s$. Then
\[
XT^+(\varkappa f_s+\varkappa^{-1}f_a)=
T_0f_s+T_1f_a=T^+f\,.\qedhere
\]
\end{proof}
Under conditions of Proposition~\ref{prop:TTp},
$T^\prime_0=T_0$ gives
$\iota T_0=XT_0X^{-1}$, and then
$\res(\iota T_0)=\rho_V=\res T_0$.
\subsubsection{Weyl function}
It is rather straightforward to verify that
the Weyl function corresponding to
a scaled \obt $(\fL,\varkappa^{-1}\Gamma_0,
\varkappa\Gamma_1)$ for $T^+$ is given by
$M_{\Gamma^\prime}=\varkappa^2M_\Gamma$ on
$\res T_0$. Assume that $X$ is as defined above
and that the conditions of Proposition~\ref{prop:TTp}
are satisfied. In order
to demonstrate the equality
$M_{\Gamma^\prime}=\varkappa^2M_\Gamma$
by using \eqref{eq:solvef00}
it suffices to convince oneself that
\[
f_0=(\varkappa X-I)f_z=
(\varkappa^2-1)P_sf_z\in\dom T_0
\]
solves \eqref{eq:solvef0}.
\subsubsection{Momentum operator}
In \cite[Section~5.2]{Hassi13} the momentum operator
$-\img\dd/\dd x$
is considered in the Hilbert space $\fH=L^2(\bbR;\fL)$.
The operators $T_\Theta$ and $T^\prime_\Theta$
associated with the \obt\!\!'s $(\fL,\Gamma_0,\Gamma_1)$
and $(\fL,\varkappa^{-1}\Gamma_0,\varkappa\Gamma_1)$,
with $\varkappa=\sqrt{3}$, are
parametrized by using $\Theta=\pm\img I_\fL$.
Our main message here is that $\fH_s$
(resp. $\fH_a$)
can be taken, for instance, as the space
of symmetric (resp. antisymmetric) $L^2$-functions.
Then the operator $T^\prime_\Theta=T_{\Theta^\prime}$
is similar to $\iota T_\Theta$ via
\[
Xf(x)=(2\varkappa)^{-1}
[(\varkappa^2+1)f(x)+(\varkappa^2-1)f(-x)]\,,
\quad f\in L^2(\bbR;\fL)\,.
\]
The example is easy to modify into a Krein space
setting with $J=-I$.
Then $\Pi_\Gamma$ is an \obt for $T^+=T^*$ provided
one replaces \eg $\Gamma_0$ by $-\Gamma_0$.
This results in $\gamma_\Gamma$ and $M_\Gamma$
with opposite signs.
\section{Transformation of a boundary pair
by means of
\texorpdfstring{$\Gamma\lto V\Gamma$}{}}\label{sec:VG}
In this section $\fL$, $\cH$ are Hilbert spaces,
$\fH$ is a Krein space with a fundamental symmetry $J$.

With a standard unitary operator
$V\co\whJ^\circ_\fL\lto\whJ^\circ_\cH$,
the present transformation scheme allows one
to consider transformed boundary pairs associated
with the same $T\in\wtscrC_s(\fH)$; see \eg
\cite[Sec.~2.5]{Behrndt20a},
\cite[Lemma~3.5]{Behrndt11},
\cite[Propositions~3.11 and 3.18]{Derkach09},
and references therein.
(But we recall that Theorem~\ref{thm:IUBP}
too allows one to construct transformed boundary pairs
associated with the same $T$; Section~\ref{exam:iota}.)
When $\fH$ is a Hilbert space, $\fL=\cH$, and $V$
is an injective isometric operator in
a $\whJ^\circ_\cH$-space, the transformed
boundary pairs are extensively studied in
\cite{Derkach17}, as they have direct applications
to boundary value problems of differential operators.
Our results for this type of transformation
are presented in Section~\ref{sec:exV}.

We begin with somewhat general cases where we
can take advantage of Lemmas~\ref{lem:Torth}
and \ref{lem:Wie}.
A variant of an \ibp for $T^+$ constructed from another
\ibp for $T^+$ is as follows.
\begin{thm}\label{thm:IBP-0}
Let $T\in\wtscrC_s(\fH)$, let
$(\fL,\Gamma)$ be an \ibp for $T^+$ with
Weyl family $M_\Gamma$, let $V$ be an
isometric relation
$\whJ^\circ_\fL\lto\whJ^\circ_\cH$, and
such that:
\begin{enumerate}[label=(\alph*),
ref=\alph*]
\item\label{item:IBP-0}
$\ol{\Gamma^{-1}\vrt_{\dom V}}=
\ol{\Gamma}^{\;-1}\vrt_{\ol{\dom}V}$ and
\item
$\ol{\dom}V\hsum\ran\ol{\Gamma}\in\wtscrC(\fL)$ and
\item\label{item:IBP-2}
$\mul V^+\mcap\dom\Gamma^+\subseteq\ker\Gamma^+$.
\end{enumerate}
Put $\Gamma^\prime:=V\Gamma$.
Then
$(\cH,\Gamma^\prime)$ is an \ibp for $T^+$
with the Weyl family
\[
M_{\Gamma^\prime}(z)=V(M_\Gamma(z))\,,\quad
z\in\bbC_*\,.
\]
\end{thm}
\begin{proof}
Clearly $\Gamma^\prime$ is an isometric relation
from a $\whJ_\fH$-space to a $\whJ^\circ_\cH$-space.
Put
\[
A^\prime_*:=\dom\Gamma^\prime=\Gamma^{-1}(\dom V)\,.
\]
By hypotheses \eqref{item:IBP-0}--\eqref{item:IBP-2}
and Lemma~\ref{lem:Torth}-\ref{item:Torth-1o},
the adjoint
\[
A^{\prime\,+}_*=\Gamma^+(\mul V^+)=
\Gamma^+(\mul V^+\mcap\dom\Gamma^+)=T
\]
and then the closure
$\ol{A^\prime_*}=\ol{A_*}=T^+$.
Thus $(\cH,\Gamma^\prime)$ is an \ibp for
$T^+$.

By definition, the Weyl family
\[
M_{\Gamma^\prime}(z):=\Gamma^\prime(\whfN_z(A^\prime_*))=
V(\Gamma(\whfN_z(A^\prime_*))\mcap\dom V)\,,\quad z\in\bbC_*\,.
\]
Because
\[
\whfN_z(A^\prime_*)=\{\whf_z\in\whfN_z(A_*)\vrt
(\exists\whl\in\dom V)\;(\whf_z,\whl)\in\Gamma \}
\]
it follows that
\[
\Gamma(\whfN_z(A^\prime_*))=
(M_\Gamma(z)\mcap\dom V)\hsum\mul\Gamma
\]
and then
\[
\Gamma(\whfN_z(A^\prime_*))\mcap\dom V=
M_\Gamma(z)\mcap\dom V
\]
since $\mul\Gamma\subseteq M_\Gamma(z)$.
Subsequently
\[
M_{\Gamma^\prime}(z)=V(M_\Gamma(z)\mcap\dom V)=
V(M_\Gamma(z))
\]
and this completes the proof.
\end{proof}
When $\dom V\supseteq\ran\Gamma$,
conditions \eqref{item:IBP-0}--\eqref{item:IBP-2} hold,
and we deduce the next corollary, with
$A^\prime_*=A_*$.
\begin{cor}\label{cor:IBP-0x}
Let $T\in\wtscrC_s(\fH)$, let
$(\fL,\Gamma)$ be an \ibp for $T^+$
with Weyl family $M_\Gamma$, let
$V$ be an isometric relation
$\whJ^\circ_\fL\lto\whJ^\circ_\cH$,
and such that $\dom V\supseteq\ran\Gamma$.
Then $(\cH,\Gamma^\prime)$ is an \ibp for
$T^+$ with the Weyl family $M_{\Gamma^\prime}$
as in Theorem~\ref{thm:IBP-0}.
\end{cor}
An analogue of Theorem~\ref{thm:IUBP}
is in part due to Corollary~\ref{cor:IBP-0x}.
\begin{thm}\label{thm:IUBP2xxcor}
Let $T\in\wtscrC_s(\fH)$, let
$(\fL,\Gamma)$ be an \ibp\!\!/\eeubp for $T^+$
with Weyl family $M_\Gamma$, and let
$V$ be a unitary relation
$\whJ^\circ_\fL\lto\whJ^\circ_\cH$
with closed $\dom V\supseteq \ran\Gamma$.
Put $\Gamma^\prime:=V\Gamma$.
Then $(\cH,\Gamma^\prime)$ is an \ibp\!\!/\eeubp
for $T^+$ with the Weyl family $M_{\Gamma^\prime}$
as in Theorem~\ref{thm:IBP-0}.

If moreover $\ker V\subseteq\mul\Gamma$
then $V$ establishes a 1-1
correspondence between an \ibp\!\!/\eeubp $(\fL,\Gamma)$
for $T^+$
and an \ibp\!\!/\eeubp $(\cH,\Gamma^\prime)$
for $T^+$.
\end{thm}
\begin{proof}
Corollary~\ref{cor:IBP-0x} applies to an \ibp $(\fL,\Gamma)$
for $T^+$.

Suppose $(\fL,\Gamma)$ is an \eeubp for $T^+$;
hence $\ol{\Gamma}=\Gamma_\#$,
$\Gamma^{\prime\,+}=(V\ol{\Gamma})^{-1}$.
Since $\ran\ol{\Gamma}\subseteq\dom V$,
Lemma~\ref{lem:Derk} gives
$(V\ol{\Gamma})^+=(V\ol{\Gamma})^{-1}$, so
$(\cH,\Gamma^\prime)$ is an \eeubp for $T^+$.

Finally, we verify that
an \ibp\!\!/\eeubp $(\cH,\Gamma^\prime)$
for $T^+$ leads back to an \ibp\!\!/\eeubp
$(\fL,\Gamma)$ for $T^+$.
Since $\ker V\subseteq\mul\Gamma$, and then
$\dom V\supseteq\ran\Gamma$, it holds
\[
V^{-1}\Gamma^\prime=V^{-1}V\Gamma=
\{(\whf,\whl+\whxi)\in\Gamma\vrt
\whl\in\dom V\,;\,\whxi\in\ker V \}=\Gamma\,.
\]
By Lemma~\ref{lem:Derk}
$\Gamma^+=\Gamma^{\prime\,+}V$
and $\ol{\Gamma}=V^{-1}\ol{\Gamma^\prime}$, since
$\ran\Gamma^\prime=V(\ran\Gamma)\subseteq\ran V$
and
$\dom\Gamma^{\prime\,+}=V(\dom\Gamma^+)\subseteq\ran V$,
respectively.
\end{proof}
When $\dom V\subseteq\ran\Gamma$, we apply
Lemmas~\ref{lem:Torth}, \ref{lem:Wie} to a \ubp
In this case the transformed \bp is
associated with a generally different from $T$
relation.
\begin{thm}\label{thm:GunTp}
Let $T\in\wtscrC_s(\fH)$, let
$(\fL,\Gamma)$ be a \ubp for $T^+$
with Weyl family $M_\Gamma$, let $V$
be an isometric relation
$\whJ^\circ_\fL\lto\whJ^\circ_\cH$,
and such that
\[
\begin{split}
\dom V\subseteq\ran\Gamma\,,\quad
\mul V^+\mcap\ran\Gamma\subseteq\ol{\dom}V
\mcap\ran\Gamma\,,
\\
\ol{\dom}V\hsum\ran\Gamma\in\wtscrC(\fL)\,,\quad
\mul V^+\hsum\ran\Gamma\in\wtscrC(\fL)\,.
\end{split}
\]
Put $T^\prime:=\Gamma^{-1}(\mul V^+)$,
$\Gamma^\prime:=V\Gamma$.
Then:
$T^\prime\in\wtscrC_s(\fH)$,
$(\cH,\Gamma^\prime)$ is an \ibp for
the adjoint relation
$T^{\prime\,+}=\Gamma^{-1}(\ol{\dom}V)$, and
the corresponding Weyl family $M_{\Gamma^\prime}$
is as in Theorem~\ref{thm:IBP-0}.
\end{thm}
\begin{proof}
With the notation as in the proof
of Theorem~\ref{thm:IBP-0}
$T^\prime$ must be equal to
\[
A^{\prime\,+}_*=\mul\Gamma^{\prime\,+}=
\Gamma^{-1}(\dom V)^+\,.
\]
By applying Lemma~\ref{lem:Wie}
condition \eqref{item:IBP-0} holds.
Since $\ol{\dom}V\hsum\ran\Gamma$ is closed,
by Lemma~\ref{lem:Torth}-\ref{item:Torth-1o}
$T^\prime$ is as defined in the
theorem. Now we verify that
$T^\prime\subseteq T^{\prime\,+}$.
Since $\Gamma^{-1}\vrt_{\mul V^+}$
and $\mul V^+\hsum\ran\Gamma$
are closed relations,
again by Lemma~\ref{lem:Torth}-\ref{item:Torth-1o}
$T^{\prime\,+}$ is as given in the theorem.
Finally $\mul V^+\mcap\ran\Gamma\subseteq\ol{\dom}V
\mcap\ran\Gamma$ shows that
$T^\prime\in\wtscrC(\fH)$ is symmetric in a Krein
space $\fH$.

The computation of $M_{\Gamma^\prime}$
is as in the proof of Theorem~\ref{thm:IBP-0}.
\end{proof}
When additionally $\Gamma$ is surjective,
a \ubp $(\fL,\Gamma)$ turns into an \obt $\Pi_\Gamma$
for $T^+$. In that case $\mul\Gamma^\prime=\mul V$, so
that $\Gamma^\prime$ is not necessarily an operator.
\begin{cor}\label{cor:GunTp}
Let $T\in\wtscrC_s(\fH)$, let
$\Pi_\Gamma=(\fL,\Gamma_0,\Gamma_1)$ be an \obt
for $T^+$ with Weyl family $M_\Gamma$, and let
$V$ be an isometric relation
$\whJ^\circ_\fL\lto\whJ^\circ_\cH$
such that $\mul V^+\subseteq\ol{\dom}V$.
Then $(\cH,\Gamma^\prime)$ is an \ibp for $T^{\prime\,+}$
with the Weyl family $M_{\Gamma^\prime}$ as in
Theorem~\ref{thm:IBP-0}.
\end{cor}
\begin{rem}\label{rem:1}
In Theorem~\ref{thm:GunTp} and Corollary~\ref{cor:GunTp}
a closed symmetric $T^\prime\in\mrm{Ext}_T$ and
$\ran\Gamma^\prime=\ran V$.
Particularly, if $V$
is an injective isometric operator
with dense range and satisfying
$\mul V^+\subseteq\dom V$, then
by \cite[Proposition~2.5]{Derkach06}
$\mul V^+=\{0\}$, so $T^\prime=T$.
We will refer to this remark thereafter.
\end{rem}
\subsection{Transformed \texorpdfstring{\ibt}{} }
Using
$\mul\Gamma^\prime=V(\mul\Gamma)$,
if $\Pi_\Gamma=(\fL,\Gamma_0,\Gamma_1)$
is an \ibt for $T^+$ and $V$ is an
isometric operator
$\whJ^\circ_\fL\lto\whJ^\circ_\cH$,
then $\Gamma^\prime$ is an isometric operator
$A^\prime_*:=\dom\Gamma^\prime\ni\whf\mapsto
(\Gamma^\prime_0\whf,\Gamma^\prime_1\whf)$ with
\begin{equation}
A^\prime_*=\{\whf\in A_*\vrt
(\Gamma_0\whf,\Gamma_1\whf)\in\dom V \}
\label{eq:Apdom}
\end{equation}
and (in matrix notation)
\begin{equation}
\begin{pmatrix}\Gamma^\prime_0 \\
\Gamma^\prime_1 \end{pmatrix}:=
V\begin{pmatrix}\Gamma_0 \\
\Gamma_1 \end{pmatrix}\quad\text{on}\quad
A^\prime_*\,.
\label{eq:Apdom2}
\end{equation}
Specifically
$A^\prime_*=A_*$ if $\dom V\supseteq\ran\Gamma$
(Corollary~\ref{cor:IBP-0x},
Theorem~\ref{thm:IUBP2xxcor}).

Similar to $T_0$ we let
$T^\prime_0:=\ker\Gamma^\prime_0$, where generally
$\Gamma^\prime_0:=\pi_0(\Gamma^\prime)$.

Given an \ibt $\Pi_\Gamma$ for $T^+$,
under suitable conditions we specify
the transformed \ibt $\Pi_{\Gamma^\prime}$ for $T^+$
in somewhat more detail by employing a different
argument than previously. Namely
\cite[Remark~7.7]{Derkach12},
if $\Gamma^\prime\co\whJ_\fH\lto\whJ^\circ_\cH$
is isometric, if
$\ker\Gamma^{\prime\,+}=\mul\Gamma^\prime$,
and if there exists a closed relation
$\fX\subseteq\fX^+\subseteq\dom\Gamma^\prime$
($\fX^+$ is the adjoint in a Krein space $\fH$),
then $\mul\Gamma^{\prime\,+}=\ker\Gamma^\prime
\subseteq\fX$;
we stress that $\fX$ need not be closed
for showing $\mul\Gamma^{\prime\,+}=\ker\Gamma^\prime$,
but it has to be closed
for having $\ker\Gamma^\prime\subseteq\fX$.
\begin{lem}\label{lem:VVV}
Let $T\in\wtscrC_s(\fH)$, let
$\Pi_\Gamma=(\fL,\Gamma_0,\Gamma_1)$
be an \ibt for $T^+$ with Weyl family
$M_\Gamma$. Assume that $\ker\Gamma=T$
and that $T_0$ is self-adjoint.
Define $\Gamma^\prime:=V\Gamma$, where
$V$ is an injective isometric operator
$\whJ^\circ_\fL\lto\whJ^\circ_\cH$
such that:
\begin{enumerate}[label=(\alph*),
ref=\alph*]
\item\label{item:VVV-a}
$\Gamma_1(T_0)\subseteq \mul\dom V$,
\item
$V(\ran\Gamma)$ is dense in $\cH^2$,
\item\label{item:VVV-c}
$T^\prime_0=T_0$.
\end{enumerate}
Then
$\Pi_{\Gamma^\prime}=(\cH,\Gamma^\prime_0,
\Gamma^\prime_1)$
is an \ibt for $T^+$
such that $\ker\Gamma^\prime=T$.
The corresponding Weyl family $M_{\Gamma^\prime}$
is as in Theorem~\ref{thm:IBP-0}.
\end{lem}
\begin{proof}
In this situation
$\ker\Gamma^\prime=\Gamma^{-1}(\ker V)=\ker\Gamma$,
$\ran\Gamma^\prime=V(\ran\Gamma)$,
$\fX:=T_0=T^\prime_0$, and it remains
to verify that $T_0\subseteq A^\prime_*$.
But this is seen from
\[
T_0\mcap A^\prime_*=
\{\whf\in T_0\vrt \Gamma_1\whf\in\mul\dom V \}
\]
by applying hypothesis \eqref{item:VVV-a}.
\end{proof}
In a Hilbert space setting
an \ibt $\Pi_{\Gamma^\prime}$
for $T^*$ with the properties as in Lemma~\ref{lem:VVV},
namely, $\Gamma^\prime$ has
dense range and $T^\prime_0$ is self-adjoint,
is known as a \qbt
(\cite[Definition~2.1]{Behrndt07}).
If moreover $\ran\Gamma^\prime_0=\cH$ then
a \qbt becomes a $B$-generalized \bt\!\!,
for which $\Gamma^\prime$ becomes unitary.
\begin{rem}\label{rem:2}
Conditions \eqref{item:VVV-a} and \eqref{item:VVV-c} hold,
for example, in the following situation.
Let $\fH$, $\cH$, $\fL$ be linear sets.
Let $\Gamma$ be a relation $\fH^2\lto\fL^2$,
$V$ a relation $\fL^2\lto\cH^2$, and let
\[
V_*:=\dom\bigl( V\mcap(\fL^2\times(\{0\}\times\cH ))
\bigr)
\]
such that
\[
\dom V_*=\{0\}\,,\quad
\Gamma_1(T_0)\subseteq\mul V_*\,.
\]
Put $\Gamma^\prime:=V\Gamma$. Then
$T^\prime_0=T_0$:
Indeed,
by definition $T^\prime_0=\Gamma^{-1}(V_*)$,
so by hypothesis $\dom V_*=\{0\}$ it holds
\[
T^\prime_0=\{\whf\in T_0\vrt
(\exists l^\prime\in\mul V_*)\;
(\whf,l^\prime)\in\Gamma_1 \}
\]
and then $T^\prime_0=T_0$ by hypothesis
$\Gamma_1(T_0)\subseteq\mul V_*$.
For example, under conditions of Theorem~\ref{thm:GunTp},
if in addition $V$ is as above
and has dense range,
then $\Gamma^\prime$ has dense range
and $T^\prime_0=T_0$. If moreover $\ran\Gamma_0$
is a subspace of $\fL$ then by
Proposition~\ref{prop:equivfNTh}
$T_0$ is self-adjoint.
\end{rem}
In closing remarks we compare
Corollary~\ref{cor:GunTp} with Lemma~\ref{lem:VVV}.
Under conditions of Corollary~\ref{cor:GunTp}
suppose in addition that $V$ is an injective
operator with dense range and such that
$\mul V^+\subseteq\dom V$. Then
$\mul V^+=\{0\}$ and $T^\prime=T$;
see Remark~\ref{rem:1} after Corollary~\ref{cor:GunTp}.
Thus, $\Pi_{\Gamma^\prime}$ is an \ibt
for $T^+$ with dense $\ran\Gamma^\prime$.
If moreover $V$ satisfies the conditions in
Remark~\ref{rem:2} then $T^\prime_0=T_0$
is self-adjoint
(and then putting $J=I$, $\Pi_{\Gamma^\prime}$
is a \qbt for $T^*$).
On the other hand, if $\Pi_\Gamma$ is an \obt for $T^+$
and $V$ is an injective isometric operator
with dense range
and satisfying the conditions in
Remark~\ref{rem:2}, then by
Lemma~\ref{lem:VVV} $\Pi_{\Gamma^\prime}$
is an \ibt for $T^+$ with dense $\ran\Gamma^\prime$
and self-adjoint $T^\prime_0=T_0$.
One sees that Corollary~\ref{cor:GunTp}
demands for a seemingly additional assumption
$\mul V^+\subseteq\dom V$ (and then $\mul V^+=\{0\}$)
for the same conclusion. The next proposition
shows, however, that $\mul V^+=\{0\}$ follows
from the rest of the properties of $V$ and $\Gamma$.
\begin{prop}\label{prop:mulV+}
Under hypotheses of Lemma~\ref{lem:VVV}
condition~\eqref{item:IBP-2} in Theorem~\ref{thm:IBP-0} holds.
\end{prop}
\begin{proof}
This is because by Lemma~\ref{lem:VVV}
\[
T^\prime=\Gamma^{-1}(\dom V)^+=T\,.
\]
Then
\[
T=\Gamma^{-1}(\dom V)^+\supseteq
\Gamma^+(\mul V^+)\supseteq\mul\Gamma^+=T
\]
and it follows that
\[
T=\Gamma^+(\mul V^+)\quad\text{hence}\quad
\mul V^+\mcap\dom\Gamma^+\subseteq\ker\Gamma^+\,.\qedhere
\]
\end{proof}
Therefore,
if $\Pi_\Gamma$ is an \obt for $T^+$, if
$V$ is an injective isometric operator
with dense range and satisfying the conditions in
Remark~\ref{rem:2},
then $\mul V^+=\{0\}$, so
Corollary~\ref{cor:GunTp} and Lemma~\ref{lem:VVV}
coincide.
\begin{rem}
A sufficient condition to have
$\mul V^+=\{0\}$ for an injective isometric operator $V$
with dense range is $V^*_*\subseteq\dom V$,
where $V^*_*$ denotes
the adjoint in $\fL$ of $V_*\in\scrC(\fL)$.
The argument uses \cite[Remark~7.7]{Derkach12}, since
$V_*$ is symmetric in $\fL$,
\[
V_*=V^{-1}(\{0\}\times\cH)\subseteq
V^+(\{0\}\times\cH)\subseteq V^*_*\,.
\]
Note that the condition
$V^*_*\subseteq\dom V$ is stronger than
$\mul V^+\subseteq\dom V$.
We emphasize that $V_*$ need not be
closed in order to apply the proof of
\cite[Eq.~(7.9)]{Derkach12}.
Consider, further,
an operator $V$ in matrix form
$\bigl(\begin{smallmatrix}
A & B \\ C & D\end{smallmatrix}\bigr)$ for
some (generally) unbounded operators $A,\ldots,D$ from
$\fL$ to $\cH$. Since $\dom V_*=\{0\}$, it follows
that $V_*$ is of the form $\{0\}\times\fN_*$, with
$\fN_*:=\ker B\mcap\dom D$.
This shows that $V^*_*\subseteq\dom V$ iff
(the orthogonal complement in $\fL$)
$\fN^\bot_*\subseteq\dom A\mcap\dom C$ and
$B$, $D\in[\fL,\cH]$. Note that, for instance,
$B$ can be trivial; for $B=0$ and
a bounded $D$, $V_*=\{0\}\times\fL$ is self-adjoint
in $\fL$ and $A$ is necessarily injective.
A slight generalization of the present model
is considered in the next section.
\end{rem}
\subsection{Example: Constructing a
\texorpdfstring{\qbt}{}from
an \texorpdfstring{\ibt}{} }\label{sec:exV}
We assume that $\fH$ is a Hilbert space and
$\fL=\cH$. In \cite[Section~3.4]{Derkach17} the following
injective isometric operator in
a $\whJ^\circ_\cH$-space is defined:
\begin{equation}
V:=\begin{pmatrix}
G^{-1} & 0 \\ EG^{-1} & G^*
\end{pmatrix}\,.
\label{eq:DerV}
\end{equation}
Whenever we refer to $V$ in
\eqref{eq:DerV}, we assume that the
operators $E\subseteq E^*$ and $G$
are densely defined in $\cH$,
$G$ is injective and has dense range, so that
the adjoint $G^*$ is also an injective operator,
which is not necessarily densely defined.

By definition
\begin{align*}
\dom V=&\ran G_E\times\dom G^*\,,
\quad G_E:=G\vrt_{\dom E}\,,
\\
\ran V=&E_G\hsum(\{0\}\times\ran G^*)\,,
\quad E_G:=E\vrt_{\dom G}
\end{align*}
and subsequently
\begin{equation}
\mul V^+=\mul\ol{G}\times\ker G^*_E\,,\quad
\ker V^+=E^*_G\vrt_{\ker\ol{G}}\,.
\label{eq:kerV+}
\end{equation}
Note that $\dom E_G=\dom G_E$.
If in addition $V$ in \eqref{eq:DerV} satisfies
\begin{equation}
E^*_G\vrt_{\ker\ol{G}}=\{0\}\,,\quad
\mul\ol{G}\subseteq\ran G_E\,,\quad
\ker G^*_E\subseteq\dom G^*
\label{eq:V*opts}
\end{equation}
then $\ol{\dom}G^*=\ol{\ran}G_E=\cH$;
recall Remark~\ref{rem:1}.

Let $\Pi_\Gamma=(\cH,\Gamma_0,\Gamma_1)$
be an \ibt for $T^*$ such that
$\ker\Gamma=T$, and let $\Gamma^\prime:=V\Gamma$.
In \cite[Lemma~3.12]{Derkach17} it is stated that
$\Pi_{\Gamma^\prime}=(\cH,\Gamma^\prime_0,\Gamma^\prime_1)$
is an \ibt with domain $A^\prime_*$ and
kernel $\ker\Gamma^\prime=T$; see
\eqref{eq:Apdom}, \eqref{eq:Apdom2}.
From this statement it is not clear to us
to which closed symmetric relation the triple
$\Pi_{\Gamma^\prime}$ is associated.
On the one hand, by definition, $\Pi_{\Gamma^\prime}$
must be associated with $T^\prime=\mul\Gamma^{\prime\,+}$.
But on the other hand, since generally
$\ker\Gamma^\prime=\Gamma^{-1}(\ker V)$,
and then particularly $\ker\Gamma^\prime=\ker\Gamma=T$,
we do not see how, under the hypotheses
imposed on $V$ in \eqref{eq:DerV},
$\ker\Gamma^\prime=T$ would suffice for
showing that $T^\prime$ is symmetric in $\fH$.
Let us recap that
\[
A^\prime_*=\Gamma^{-1}(\dom V)
\subseteq A_*
\]
is not necessarily dense
in $\ol{A_*}=T^*$, as generally $T^\prime\supseteq T$.
Under additional hypotheses, however, we can
achieve that $T^\prime\subseteq T^{\prime\,*}$.

The first part of the next proposition
is the direct application of Corollary~\ref{cor:GunTp},
while the last statement follows from
Lemma~\ref{lem:VVV} or from
Corollary~\ref{cor:GunTp},
Remark~\ref{rem:2} with $V_*=\{0\}\times\dom G^*$,
and Proposition~\ref{prop:mulV+};
see the remark after Proposition~\ref{prop:mulV+}.
\begin{prop}\label{prop:VVV}
Let $T\in\wtscrC_s(\fH)$, let
$\Pi_\Gamma=(\cH,\Gamma_0,\Gamma_1)$ be an \obt for $T^*$
with Weyl function $M_\Gamma$. Let $V$ be as in \eqref{eq:DerV}
and suppose in addition that
\[
\mul\ol{G}\subseteq\ol{\ran}G_E\,.
\]
Then $\Pi_{\Gamma^\prime}=
(\cH,\Gamma^\prime_0,\Gamma^\prime_1)$, where
\[
\Gamma^\prime_0:=G^{-1}_E\Gamma_0\,,\quad
\Gamma^\prime_1:=E_GG^{-1}_E\Gamma_0+G^*\Gamma_1
\]
\[
\text{on }
A^\prime_*=\{\whf\in T^*\vrt
\Gamma_0\whf\in\ran G_E\,;\,
\Gamma_1\whf\in\dom G^* \}\,,
\]
is an \ibt for the adjoint
\[
T^{\prime\,*}=\{\whf\in T^*\vrt
\Gamma_0\whf\in\ol{\ran} G_E\,;\,
\Gamma_1\whf\in\ol{\dom} G^* \}
\]
of a closed symmetric extension
\[
T^\prime=\{\whf\in T^*\vrt
\Gamma_0\whf\in\mul\ol{G}\,;\,
\Gamma_1\whf\in\ker G^*_E \}
\]
of $T$. The Weyl function $M_{\Gamma^\prime}$
corresponding to $\Pi_{\Gamma^\prime}$ is given by
\[
M_{\Gamma^\prime}(z)=E_G+G^*M^0_\Gamma(z)G_E\,,
\quad
M^0_\Gamma(z):=M_\Gamma(z)\vrt_{\ran G_E}\,,
\]
\begin{align*}
\dom M_{\Gamma^\prime}(z)=&
\{h\in\dom E_G\vrt M^0_\Gamma(z)G_Eh\in\dom G^*\}
\subseteq\ran\Gamma^\prime_0
\\
=&\{h\in\dom E_G\vrt(\exists\whf_0\in T_0)\;
\Gamma_1\whf_0+M^0_\Gamma(z)G_Eh\in\dom G^*\}
\end{align*}
for $z\in\bbC_*$.

If in particular $V$ in \eqref{eq:DerV} satisfies
\[
E^*_G\vrt_{\ker\ol{G}}=\{0\}\,,\quad
\Gamma_1(T_0)\subseteq\dom G^*
\]
then $\ol{\dom}G^*=\ol{\ran}G_E=\cH$,
$\Pi_{\Gamma^\prime}$ is a \qbt for $T^{*}$,
and
$\dom M_{\Gamma^\prime}(z)=\ran\Gamma^\prime_0$.
\end{prop}
When $\Pi_\Gamma$ is a suitable \ibt for $T^*$,
but is not necessarily an \obt\!\!,
under a stronger assumption
$\ol{G}\in[\cH]$ we use
Lemma~\ref{lem:VVV} and Remark~\ref{rem:2}
to construct a \qbt for $T^*$;
see Theorem~\ref{thm:VVV} below.

We put
\[
G_0:=G\vrt_{G^{-1}(\dom\Theta_0)}\,,\quad
E_0:=E\vrt_{G^{-1}(\dom\Theta_0)}\,,\quad
\Theta_0:=(\ran\Gamma)\vrt_{\ran G_E}
\]
so that the inclusions
in $G_0\subseteq G_E$, $E_0\subseteq E_G$,
become the equalities
for $\ran\Gamma_0=\cH$. Note that
$\dom E_0=\dom G_0$. If $T_0$ is self-adjoint
and $M^0_\Gamma$ is as in Proposition~\ref{prop:VVV},
then by the von Neumann formula the relation
$\Theta_0$ can be represented in the form
\[
\Theta_0=M^0_\Gamma(z)\hsum
(\{0\}\times\Gamma_1(T_0))
\]
for each fixed $z\in\bbC_*$.

First we prove the lemma from which it would
follow that, under hypotheses imposed on $V$
in Theorem~\ref{thm:VVV}, $\Gamma^\prime$
has dense range.
\begin{lem}\label{lem:QBTex}
Let $\Gamma=\{(\whf,(\Gamma_0\whf,\Gamma_1\whf))\vrt
\whf\in A_* \}$ for some operators $\Gamma_0$, $\Gamma_1$
from $A_*$ to $\cH$. Let $V$ be as in \eqref{eq:DerV}
with $\ol{G}\in[\cH]$. Put
$\Gamma^\prime:=V\Gamma$.
Then
\[
\ker\Gamma^{\prime\,+}=
E^*_0+G^*\Theta^*_0\ol{G}\,.
\]
\end{lem}
\begin{proof}
By definition and hypothesis
$\dom G^*=\cH$ the operator
\[
\Gamma^\prime=
\{(\whf,(G^{-1}_0\Gamma_0\whf,E_0G^{-1}_0\Gamma_0\whf+
G^*\Gamma_1\whf))\vrt \whf\in A^\prime_* \}\,,
\]
\[
A^\prime_*=
\{\whf\in A_*\vrt
(\Gamma_0\whf,\Gamma_1\whf)\in\Theta_0 \}\,.
\]
The boundary condition for
$\whf\in A^\prime_*$ implies that
$(\exists(r,s)\in\Theta^*_0)$
\[
\braket{\Gamma_0\whf,s}_\cH=
\braket{\Gamma_1\whf,r}_\cH\,.
\]
Put $s_1:=G^*s$; then
$(r,s)\in\Theta^*_0$ $\Leftrightarrow$
$(r,s_1)\in G^*\Theta^*_0$. Because
$s=G^{*\,-1}s_1$ and
\[
\Gamma_0(A^\prime_*)=\dom\Theta_0=\ran G_0
\]
it follows that
\[
\braket{\Gamma_0\whf,s}_\cH=
\braket{\Gamma_0\whf,G^{*\,-1}s_1}_\cH=
\braket{G^{-1}_0\Gamma_0\whf,s_1}_\cH\,.
\]
Therefore
\begin{equation}
(\forall\whf\in A^\prime_*)\;
(\exists(r,s_1)\in G^*\Theta^*_0)\;
\braket{G^{-1}_0\Gamma_0\whf,s_1}_\cH=
\braket{\Gamma_1\whf,r}_\cH\,.
\label{eq:bcA}
\end{equation}

The $\whJ^\circ_\cH$-orthogonal complement
$(\ran\Gamma^\prime)^{[\bot]}=\ker\Gamma^{\prime\,+}$
consists of
$(h^\prime,h)\in\cH^2$ such that
\[
(\forall\whf\in A^\prime_*)\;
\braket{G^{-1}_0\Gamma_0\whf,h}_\cH-
\braket{E_0G^{-1}_0\Gamma_0\whf,h^\prime}_\cH=
\braket{\Gamma_1\whf,\ol{G}h^\prime}_\cH\,.
\]
By \eqref{eq:bcA} this implies that
each $(h^\prime,h)\in(\ran\Gamma^\prime)^{[\bot]}$
is characterized by
$(r,s_1)\in G^*\Theta^*_0$ such that
\[
(h^\prime,h-s_1)\in
E^*_0\,,\quad
\ol{G}h^\prime-r\in
\ker\Theta^*_0\,;
\]
the second relation uses that
$\Gamma_1(A^\prime_*)=\ran\Theta_0$.

Putting $s_0:=h-s_1$ and $r_1:=r-\ol{G}h^\prime$
we conclude that $(\ran\Gamma^\prime)^{[\bot]}$
consists of $(h^\prime,s_0+s_1)$ such that
$(h^\prime,s_0)\in E^*_0$ and
$(\exists r_1\in\ker\Theta^*_0)$
$(\ol{G}h^\prime+r_1,s_1)\in G^*\Theta^*_0$;
we note that $-r_1\in\ker\Theta^*_0$ iff
$r_1\in\ker\Theta^*_0$. By noting also that
$\ker(G^*\Theta^*_0)=\ker\Theta^*_0$, and that
$(\ol{G}h^\prime,s_1)\in G^*\Theta^*_0$ iff
$(h^\prime,s_1)\in G^*\Theta^*_0\ol{G}$,
the claim
of the lemma follows as stated.
\end{proof}
As an exercise, with $\ran\Gamma=\cH^2$,
one verifies that $(\ran V)^{[\bot]}=\ker V^+$
is as in \eqref{eq:kerV+}.
\begin{thm}\label{thm:VVV}
Let $T\in\wtscrC_s(\fH)$, let
$\Pi_\Gamma=(\cH,\Gamma_0,\Gamma_1)$ be
an \ibt for $T^*$ with Weyl function $M_\Gamma$
and such that $\ker\Gamma=T$ and $T_0$
is self-adjoint.
Let $V$ be as in \eqref{eq:DerV} and suppose in
addition that $\ol{G}\in[\cH]$ and
\[
\dom E^*_0\mcap\dom(\Theta^*_0\ol{G})=\{0\}\,,
\quad
\ol{\dom}E_0=\ol{\dom}\Theta_0=\cH\,.
\]
Then $\Pi_{\Gamma^\prime}=
(\cH,\Gamma^\prime_0,\Gamma^\prime_1)$
is a \qbt for $T^*$, where
\[
\Gamma^\prime_0:=G^{-1}_0\Gamma_0\,,\quad
\Gamma^\prime_1:=E_0G^{-1}_0\Gamma_0+G^*\Gamma_1
\]
\[
\text{on }
A^\prime_*=
\{\whf\in A_*\vrt
(\Gamma_0\whf,\Gamma_1\whf)\in\Theta_0 \}\,.
\]
The Weyl function $M_{\Gamma^\prime}$
corresponding to $\Pi_{\Gamma^\prime}$ is given by
\[
M_{\Gamma^\prime}(z)=E_0+G^*M^0_\Gamma(z)G_0\,,
\quad
\dom M_{\Gamma^\prime}(z)=\dom E_0
\]
for $z\in\bbC_*$.
\end{thm}
\begin{proof}
By hypotheses and Lemma~\ref{lem:QBTex}
and Remark~\ref{rem:2},
$\ran\Gamma^\prime$ is dense in $\cH^2$
and $T^\prime_0=T_0$. Then the claim
follows from Lemma~\ref{lem:VVV}.
\end{proof}
We mention that
\[
\dom(\Theta^*_0\ol{G})=
\{h\in\cH\vrt
\ol{G}h\in\dom M^0_\Gamma(z)^*\mcap
\Gamma_1(T_0)^\bot \}\,.
\]
This suggests that, in case $\ol{G}\in[\cH]$ is injective,
the disjointness of sets
$\dom E^*_0$ and $\dom(\Theta^*_0\ol{G})$
can be controlled by the solely
condition $\ol{\Gamma_1(T_0)}=\cH$.
\begin{rem}
We make some comments
on special cases of Theorem~\ref{thm:VVV}:

$\circ$
Assume that $\Pi_\Gamma$ is an $AB$-generalized \bt
for $T^*$; that is
\cite[Definition~1.8]{Derkach17},
$\Pi_\Gamma$ is an \ibt for $T^*$ such that
$T_0$ is self-adjoint and $\ran\Gamma_0$
is dense in $\cH$. By \cite[Theorem~4.2(i)]{Derkach17}
$\ker\Gamma=T$, so we may apply Theorem~\ref{thm:VVV}.
According to \cite[Lemma~3.7(vi), Theorem~4.2(iv),
Corollary~4.3(ii)]{Derkach17} the
orthogonal complement
\[
\Gamma_1(T_0)^\bot=\ker\ol{\gamma_\Gamma(z)}
=\ker\ol{\Im M_\Gamma(z)}\,,
\quad \ol{\gamma_\Gamma(z)}\in[\cH,\fH]\,,
\quad z\in\bbC_*
\]
where $\gamma_\Gamma$ is the $\gamma$-field
corresponding to $\Pi_\Gamma$ and
the imaginary part
$\Im M_\Gamma(z):=(M_\Gamma(z)-
M_\Gamma(z)^*)/(2\img)$; notice that
$M_\Gamma(z)$ is a densely defined operator
with domain $\ran\Gamma_0$.

$\circ$
A $B$-generalized \bt\!\!,
a \qbt\!\!, an \obt\!\!, and an $S$-generalized
\bt are all special cases of an
$AB$-generalized \bt
For instance,
if $\Pi_\Gamma$ is a $B$-generalized \bt
for $T^*$, that is, an $AB$-generalized \bt
with $\ran\Gamma_0=\cH$, and if
$V$ is as in \eqref{eq:DerV} with
\begin{equation}
\ol{G}\in[\cH]\quad\text{and}\quad
\dom E^*_G\mcap\ker\ol{G}=\{0\}
\quad\text{and}\quad
\ol{\dom}E_G=\cH
\label{eq:EGlast}
\end{equation}
then $\Pi_{\Gamma^\prime}$
is a \qbt for $T^*$ as described in
Theorem~\ref{thm:VVV}.
In this case
$\ol{\dom}\Theta_0=\ol{\ran}G_E=\cH$ is automatically
satisfied by the above and \eqref{eq:V*opts}.
The present variant of Theorem~\ref{thm:VVV} also
follows from \cite[Theorem~7.9]{Wietsma12},
\cite{Wietsma13}, where
a \qbt is considered in a Krein space setting.
For $G=I_\cH$, the result is known from
\cite[Theorem~7.57]{Derkach12}, see also
\cite[Theorem~4.4]{Derkach17} with a nontrivial
$\mul\Gamma$.

$\circ$
Finally, if we let, for example,
$\Pi_\Gamma$ be an \obt for $T^*$,
then for $V$ as in \eqref{eq:DerV} and \eqref{eq:EGlast},
Theorem~\ref{thm:VVV} reduces to the last statement
of Proposition~\ref{prop:VVV};
see the remark above Eq.~(A.13) in
\cite{Wietsma12}, where the converse implication
is stated as well.
\end{rem}
\section*{Acknowledgments}
The author would like to thank the anonymous
referees for their helpful comments.



\end{document}